\pdfoutput=1
\documentclass[a4paper]{amsart}
\usepackage[utf8]{inputenc}
\usepackage[expansion=false]{microtype} 		
\usepackage{amsmath, amsfonts, amssymb, amsthm, amsopn}
\usepackage{tikz, tikz-cd}
\usepackage{accents}
\usepackage{mathtools}

\usepackage[pdfusetitle,unicode,hidelinks]{hyperref}    
\usepackage{xcolor}\definecolor{mathgreen}{HTML}{009b77}
\usepackage{pdfsync}					
\usepackage{scalerel}

\newcommand{\ie}{{\sfcode`\.1000 i.e.}}    		
\newcommand{\eg}{{\sfcode`\.1000 e.g.}}    		

\theoremstyle{plain}
\newtheorem{theorem}{Theorem}[section]			
\newtheorem*{theorem*}{Theorem}
\newtheorem{proposition}[theorem]{Proposition}
\newtheorem{corollary}[theorem]{Corollary}
\newtheorem{lemma}[theorem]{Lemma}

\theoremstyle{definition}
\newtheorem{definition}[theorem]{Definition}

\newtheorem{remark}[theorem]{Remark}
\newtheorem{construction}[theorem]{Construction}

\newtheorem{sect}[theorem]{} 

\numberwithin{equation}{section}
\hypersetup{linktoc=page}

\title{A Bloch--Ogus Theorem for henselian local rings in mixed characteristic}
\author{Johannes Schmidt}
\address[Johannes Schmidt]{\newline Mathematisches Institut, Universit\"at Heidelberg, 69120 Heidelberg, Germany}
\email{jschmidt@mathi.uni-heidelberg.de}
\author{Florian Strunk}
\address[Florian Strunk]{\newline Fakult\"at f\"ur Mathematik, Universit\"at Regensburg, 93040 Regensburg, Germany}
\email{florian.strunk@ur.de}

\renewcommand{\AA}{\mathbb{A}}

\newcommand{\NN}{\mathbb{N}}
\newcommand{\OO}{\mathcal{O}}

\newcommand{\PP}{\mathbb{P}}
\newcommand{\ZZ}{\mathbb{Z}}

\newcommand{\I}{\mathcal{I}}

\newcommand{\Sm}{\mathrm{S}\mathrm{m}}

\newcommand{\mot}{\mathrm{mot}}
\newcommand{\bur}[1]{\underaccent{\bar}{#1}}

\newcommand{\id}{\mathrm{id}}

\newcommand{\Nis}{\mathrm{Nis}}

\renewcommand{\setminus}{\smallsetminus}

\let\hom=\relax
\DeclareMathOperator{\hom}{Hom}
\DeclareMathOperator{\hhom}{\underline{\hom}}

\DeclareMathOperator{\Spec}{Spec}
\DeclareMathOperator{\Spt}{\mathrm{Spt}}

\DeclareMathOperator{\SH}{{\mathcal{S}\mathcal{H}}}

\let\lim=\relax
\DeclareMathOperator*{\lim}{lim}
\let\dim=\relax
\DeclareMathOperator*{\im}{im}
\DeclareMathOperator*{\dim}{dim}
\DeclareMathOperator*{\codim}{codim}

\DeclareMathOperator*{\colim}{colim}

\thanks{The authors are supported by the SFB/CRC 1085 \emph{Higher Invariants} (Regensburg) funded by the DFG and the DFG-Forschergruppe 1920 
\emph{Symmetrie, Geometrie und Arithmetik} (Heidelberg--Darmstadt)}

\begin{document}

\begin{abstract}
We show a conditional exactness statement for the Nisnevich Gersten complex associated to an $\AA^1$-invariant cohomology theory with Nisnevich descent for smooth schemes over a Dedekind ring with only infinite residue fields.
As an application we derive a Nisnevich analogue of the Bloch--Ogus theorem for \'etale cohomology over a henselian discrete valuation ring with infinite residue field.
\end{abstract}

\maketitle


\section{Introduction}

Given an $\AA^1$-invariant cohomology theory $E$ for smooth varieties $X$ over a field $k$ with Nisnevich descent, Colliot-Th{\'e}l{\`e}ne, Hoobler and Kahn proved in~\cite{CTHK97} the exactness of the associated Gersten complex
\begin{multline}\label{eqn:introgc}
0\to H^n(Y, E) \to \bigoplus_{z\in Y^{(0)}} H^n_{\overline{\{z\}}}(Y, E) \to \bigoplus_{z\in X^{(1)}} H^{n+1}_{\overline{\{z\}}}(Y, E) \to\cdots\\
\cdots\to \bigoplus_{z\in X^{(d-1)}} H^{n+d-1}_{\overline{\{z\}}}(Y, E)  \to \bigoplus_{z\in X^{(d)}} H^{n+d}_{\overline{\{z\}}}(Y, E) \to 0,
\end{multline}
where $Y=\Spec(\mathcal{O}_{X,x})$ is the local scheme at a point $x$ and $d$ is the dimension of $X$.
The main ingredient of their proof is a geometric presentation theorem~\cite[Theorem~3.1.1]{CTHK97} for a closed immersion $Z\hookrightarrow X$ which is due to Gabber.
If $E$ is algebraic $K$-theory, this result implies the Gersten conjecture for smooth schemes over a field, originally proved by Quillen \cite[Theorem~5.11]{Quillengersten}.
Taking $E$ as \'etale cohomology with constant torsion coefficients defined over $k$, one obtains the Bloch--Ogus theorem \cite{BlochOgus}.

\medskip

In the mixed characteristic case of a discrete valuation ring with infinite residue field, an analogue of Gabber's geometric presentation theorem for a closed immersion $Z\hookrightarrow X$ was shown in~\cite[Theorem 2.1]{SS16}.
However, there are two crucial differences to the equal characteristic case:
Firstly, one has to require that the closed subscheme $Z$ does not contain any irreducible component of the special fibre of $X$.
Secondly, the presentation is not Zariski- but only Nisnevich-local in $X$.

\medskip

In this paper, our goal is to adopt the techniques of Colliot-Th{\'e}l{\`e}ne, Hoobler and Kahn to the mixed characteristic case using the more restricted version of the presentation theorem.
Our main result is the following (see~Theorem~\ref{thm: Bloch-Ogus abstract nonsense}, below).

\begin{theorem*}
Let $S$ be a Dedekind scheme with only infinite residue fields and $E$ an $\AA^1$-invariant cohomology theory for smooth schemes of finite type over $S$ with Nisnevich descent. Let $X/S$ be such a smooth scheme of dimension $d$, $x\in X$ a point and $Y=\Spec(\mathcal{O}_{X,x}^h)$ the Henselian local scheme at $x$.
\begin{enumerate}
 \item The Gersten complex \eqref{eqn:introgc} is exact possible except at the first and third (non-trivial) spot.
 \item If for each point $x$ of $X$ the forget support map for the special fibre $Y_\sigma$
 \[
  \mathbb{R}\Gamma_{Y_\sigma}(Y_\Nis,E) \to \Gamma(Y,E)
 \]
is trivial,
then the Gersten complex is exact everywhere.
\end{enumerate}
\end{theorem*}

In presence of the second condition, one obtains the usual resolution of the Nisnevich sheafification of the cohomology given by $E$ by flabby Nisnevich sheaves.
If $E$ is algebraic $K$-theory, the theorem was known before, see~\cite{BlochGersten} and \cite{GilletLevine}.

\medskip

As an application of the theorem above, we derive the following analogue of the Bloch--Ogus theorem in mixed characteristic (see~Corollary~\ref{cor: Bloch Ogus for etale cohomology of Nisnevich local schemes}).

\begin{theorem*}
Let $Y=\Spec(\mathcal{O}_{X,x}^h)$ be a Henselian local scheme of a $d$-dimensional smooth scheme $X$ of finite type over a Henselian discrete valuation ring $\mathfrak{o}$ with infinite residue field of characteristic $p$.
Let $K$ be a locally constant constructible sheaf of $\ZZ/m$-modules for $m$ prime to $p$ on the small \'etale site of $\Spec(\mathfrak{o})$.
Then the Gersten complex
\[
 0 \to {\rm H}^n(Y_{{\rm et}},K) \to \bigoplus_{z \in Y^{(0)}} {\rm H}^n(k(z),K)\to \cdots\to \bigoplus_{z \in Y^{(d)}} {\rm H}^{n-d}(k(z),K(-d)) \to 0 .
\]
is exact. Here $H^i(k(z),-)$ denotes the Galois-cohomology of the field $k(z)$.
\end{theorem*}

We remark that in \cite{Geisser04} Geisser derived the exactness of the above Gersten complex from the Bloch--Kato-Conjecture even for a (Zariski) local scheme $Y=\Spec(\mathcal{O}_{X,x})$ but only for coefficients $K = \boldsymbol{\mu}_{m}^{\otimes r}$ where $n \leq r$.
Our method of proof is more elementary, at least if the residue and quotient field of $\mathfrak{o}$ are perfect (see Remark~\ref{remark:avoidpurity} and Remark~\ref{remark:geisser}).

\medskip

The organization of the paper is as follows. In Section~\ref{Section:preliminaries} we recall some known results on basechange of presheaves of spectra. We include a short reminder on elementary properties of the codimension for schemes not necessarily over a field. In Section~\ref{Section: coniveau filtration} we define the coniveau filtration for a spectrum with Nisnevich descent and show that the filtration quotients are flabby Nisnevich sheaves. In Section~\ref{Section: Gersten complex} the Nisnevich Gersten complex is introduced. Up to this point, the $\AA^1$-invariance property has not been used yet.
Section~\ref{section:effaceability} contains an effaceability result which makes use of $\AA^1$-invariance and the geometric presentation theorem. This leads to our main theorem.
The final Section~\ref{section:blochogus} deals with the application to \'etale cohomology with torsion coefficients.

\section{Preliminaries}\label{Section:preliminaries}

Let $S$ be a base scheme, which is always assumed to be noetherian and of finite dimension.
Let moreover $\Sm_S$ be the category of smooth schemes of finite type over $S$ and $\Spt_{S^1}(\Sm_S)$ the category of presheaves of spectra on $\Sm_S$.
For an object $X\in\Sm_S$ we define the category $\Spt_{S^1}(X_\Nis)$ analogously where $X_\Nis$ denotes the small Nisnevich site on $X$.

We consider the \emph{(stable) object-wise model structure} on $\Spt_{S^1}(\Sm_S)$.
Its homotopy category $\SH_{S^1}(\Sm_S)$ is a triangulated category with exact triangles given by the homotopy (co)fibre sequences.
The left Bousfield localization at the equivalences on Nisnevich stalks is called the \emph{(stable) Nisnevich local model structure} with a fibrant replacement functor $L_\Nis$.
Likewise, in the case of the small site $X_\Nis$, we define the (stable) Nisnevich local model structure on $\Spt_{S^1}(X_\Nis)$ analogously.
In the case of the big site $\Sm_S$, a further left Bousfield localization yields the \emph{(stable) $\AA^1$-Nisnevich local model structure} with a fibrant replacement functor $L_\mot$.
The details of this model structures play no essential role for this text and we refer to the preliminary section of \cite{SS16} for further explanation and references.
We are working in the non-localized model structure only and use the fibrant replacement functors to obtain statements about the localizations.
Hence, whenever we speak of an exact triangle or a homotopy cofibre, we mean the respective terms for the \emph{object-wise} model structure.

\subsection{Basechange}

A morphism $f\colon X\to Y$ of noetherian schemes of finite Krull dimension induces a covering preserving functor $\bar f\colon \Sm_Y\to\Sm_X$ by pullback.
Precomposition ${\bar f}_*$ with ${\bar f}$ is the right adjoint of a Quillen adjunction
\[
{\bar f\,}^*: \Spt_{S^1}(\Sm_Y)\rightleftarrows \Spt_{S^1}(\Sm_X): {\bar f}_*
\]
for each of the model structures from above (see again the preliminary section of \cite{SS16} for more details and references).
By abuse of notation, we write $f_*:={\bar f}_*$ and $f^*:={\bar f\,}^*$.
If the morphism $f$ is an object of $\Sm_Y$ itself, there is an adjunction
\[
 {\bur{f}}: \Sm_X\rightleftarrows\Sm_Y : {\bar f}
\]
with the left adjoint given by post-composition.
Again, precomposition with ${\bur{f}}$ is the right adjoint of a Quillen adjunction
\[
{\bur{f}}^*: \Spt_{S^1}(\Sm_X)\rightleftarrows \Spt_{S^1}(\Sm_Y): {\bur{f}}_*
\]
for each of the model structures.
We clearly have ${\bur{f}}_*={\bar f\,}^*=f^*$ and set $f_\sharp:={\bur{f}}^*$.

Likewise, for a morphism $f\colon X\to Y$, we obtain Quillen adjunctions
\begin{equation}\label{adjunctiononthesmallsites}
\Spt_{S^1}(X_\Nis) \underset{f^*}{\overset{f_\sharp}{\rightleftarrows}} \Spt_{S^1}(Y_\Nis) \underset{f_*}{\overset{f^*}{\rightleftarrows}} \Spt_{S^1}(X_\Nis)
\end{equation}
for the object-wise and the Nisnevich local model structure where again for the first one we have to assume that $f$ is an object of $Y_\Nis$ whereas the second always exists.
%
%

For an object $X\to Y$ in $\Sm_Y$, there is a canonical covering preserving inclusion functor $\delta_{X/Y}\colon X_{\Nis}\to\Sm_Y$.
Precomposition with this functor yields the right adjoint of a Quillen adjunction
\begin{equation}\label{restrictionadjunction}
\delta_{X/Y}^* : \Spt_{S^1}(X_{\Nis}) \rightleftarrows \Spt_{S^1}(\Sm_Y) : \delta_{X/Y,*}
\end{equation}
for the object-wise and the Nisnevich local model structures.
The inclusion functor $\delta_{X/Y}$ factorizes as
\[
 X_{\Nis}\xrightarrow{\delta_{X}} \Sm_X \xrightarrow{\bur{f}} \Sm_Y
\]
(where we set $\delta_{X}:=\delta_{X/X}$) and the adjunction~\eqref{restrictionadjunction} factorizes as
\[
\Spt_{S^1}(X_{\Nis}) \underset{\delta_{X,*}}{\overset{\delta_X^*}{\rightleftarrows}} \Spt_{S^1}(\Sm_X) \underset{f^*}{\overset{f_\sharp}{\rightleftarrows}} \Spt_{S^1}(\Sm_Y). 
\]
If moreover $X$ is an object of $Y_\Nis$, there is a commutative diagram
\begin{equation}\label{commutativediagramofsites}
\begin{tikzcd}
 X_{\Nis} \arrow[d, "{\bur{f}}"] \arrow[r, "\delta_{X}"] \arrow[dr, "\delta_{X/Y}" description] & \Sm_X \arrow[d, "{\bur{f}}"] \\
 Y_{\Nis} \arrow[r, "\delta_{Y}"] & \Sm_Y
\end{tikzcd}
\end{equation}
inducing the diagramm
\[
\begin{tikzcd}
 \Spt_{S^1}(X_{\Nis}) \arrow[d, "f_\sharp"', shift right] \arrow[r, "\delta_{X}^*", shift left] & \Spt_{S^1}(\Sm_X) \arrow[d, "f_\sharp"', shift right]\arrow[l, "\delta_{X,*}", shift left] \\
 \Spt_{S^1}(Y_{\Nis}) \arrow[u, "f^*"', shift right] \arrow[r, "\delta_{Y}^*", shift left] & \Spt_{S^1}(\Sm_Y)\arrow[l, "\delta_{Y,*}", shift left]\arrow[u, "f^*"', shift right]
\end{tikzcd}
\]
of Quillen adjunctions with diagonal \eqref{restrictionadjunction}.
In particular, for an \'etale morphism $f\colon X\to Y$ the restriction $\delta_*$ to the respective small sites commutes with $f^*$.
In particular, these observations imply the following lemma.

\begin{lemma}\label{lem: restrictions to small sites vs restrictions in small sites}
Let $X\in\Sm_S$ and $g\colon \tilde{X} \to X$ \'{e}tale.
Then $g^\ast \delta_{X/S,\ast} \cong \delta_{\tilde{X}/S,\ast}$.
In particular, the unit of the adjunction $g^\ast \dashv g_\ast$ induces a canonical map 
\[
\delta_{X/S,\ast}E \to g_\ast \delta_{\tilde{X}/S,\ast}E. 
\]
for $E\in\Spt_{S^1}(\Sm_S)$.
\end{lemma}



\begin{remark}
Let $f\colon X\to Y$ be any morphism between noetherian schemes of finite Krull dimension.
A diagram analogous to~\eqref{commutativediagramofsites} with $\bar f$ in place of $\bur f$ shows that the restriction $\delta_*$ to the respective small sites commutes with $f_*$.
\end{remark}

\begin{definition}\label{defi:homotopygroups}
Let $E\in\Spt_{S^1}(\Sm_S)$ or $E\in\Spt_{S^1}(X_\Nis)$ be a spectrum. For $n\in\ZZ$, one sets $E^n(X):=\pi_{-n}(E(X))$.
\end{definition}

\begin{lemma}\label{basechangeonhomotopygroups}
Let $E\in\Spt_{S^1}(X_\Nis)$ be a spectrum, $z\in X$ a point and consider the canonical morphism $\mathfrak{z}\colon \Spec(\OO_{X,z})\to X$.
Then the canonical morphism
\[
\mathfrak{z}_*\mathfrak{z}^*\pi_n(E) \to \pi_n(\mathfrak{z}_*\mathfrak{z}^* E)
\]
is an isomorphism of presheaves on $X_\Nis$ for every $n\in\ZZ$.
\end{lemma}
\begin{proof}
For any morphism $f$, we have an isomorphism $f_*\pi_n(E)\cong \pi_n(f_*E)$ as $\pi_n$ and $f_*$ are defined object-wise.
Suppose for a moment that $\mathfrak{z}$ were an object $f\colon U\to X$ of the site $X_{\Nis}$.
In this case $f^*(F)\cong F\times_X U$.
Since homotopy presheaves and pullbacks of presheaves are calculated object-wise, we obtain
\[
f^*\pi_0(E)\cong \pi_0(E)\times_X U \cong \pi_0(E\times_X U) \cong \pi_0(f^*E).
\]
We may write $\pi_n(F)$ alternatively as $\pi_0(\hhom(S^n,F))$, where $\hhom$ denotes the internal mapping space of preshaves.
For this internal mapping space, there are isomorphisms
\[
f_*(f^*\hhom(S^n,E)) \xrightarrow{\cong} f_*(\hhom(f^*(S^n),f^*(E))) \xrightarrow{\cong} \hhom(S^n,f_*f^*(E))
\]
where for the first we used that $f$ was assumed to be in $X_{\Nis}$.
Alltogether, we have
\[
\begin{array}{rcl}
f_*f^*\pi_n(E) &\cong & f_*f^*\pi_0\hhom(S^n,E)\\
               &\cong & \pi_0f_*f^*\hhom(S^n,E)\\
               &\cong & \pi_0 \hhom(S^n,f_*f^*(E))\\
               &\cong & \pi_n(f_*f^* E)
\end{array}
\]
in the case of $f$ being an object of the site $X_{\Nis}$.

For the case of the essentially open immersion $\mathfrak{z}\colon \Spec(\OO_{X,z})\to X$ of the lemma, we write $\mathfrak{z}$ as the cofiltered limit of the diagram $D_{\scaleobj{0.7}{(-)}}\colon\I\to X_{\Nis}$ given by the affine Zariski neighbourhoods of $z$ in $X$.
Then by the proof of \cite[Lemma~1.5]{SS16} (and after choosing a cofinal subdiagram) one has a canonical natural isomorphism
\[
\mathfrak{z}_*\mathfrak{z}^*(F) \cong \colim_{i\in\I} d_{i,*}d_i^*(F)
\]
where $d_i\colon D_i\to X$ is the structural morphism which is an open immersion.
The result now follows from the case handled above and from the fact that homotopy groups commute with filtered colimits.
\end{proof}

\subsection{Codimension}

In this subsection, we recall basic notations on the codimension for the convenience of the reader.

\smallskip

Let $X$ be a scheme and $Z\subseteq X$ an irreducible closed subset.
Define
\[
 \codim(Z,X) :=\operatorname{sup}\{s\in\NN_{\infty}\mid Z = Z_s\subsetneq \ldots \subsetneq Z_0\subseteq X\text{ with $Z_i\subseteq X$ irred. cl.}\}.
\]
For an arbitrary closed subset $Z\subseteq X$ we set 
\[
\codim(Z,X):=\operatorname{inf}\{\codim(Z',X)\mid \text{$Z'\subseteq Z$ is an irreducible component}\},
\]
where by convention $\codim(\emptyset,X)=\infty$, as the \emph{codimension of $Z$ in $X$}.
One has
\[
 \codim(Z,X) = \underset{z\in Z}{\operatorname{inf}}\dim(\OO_{X,z}).
\]
If $Z$ is irreducible closed with generic point $\eta_Z$, then $\codim(Z,X) = \dim(\OO_{X,\eta_Z})$.
Recall that a scheme $X$ is called \emph{catenary}, if for every two irreducible closed subsets $Z\subseteq Z'\subseteq X$ every maximal chain $Z=Z_s\subsetneq \ldots\subsetneq Z_0=Z'$ of irreducible closed subsets has the same finite length.
Examples of such are schemes (locally) of finite type over a field or over a one dimensional noetherian domain, \eg, a discrete valuation ring.

We have the following two easy lemmas.
\begin{lemma}\label{lemma:codimensionneq}
Let $X$ be a scheme and $Z\subseteq Z'\subseteq X$ two irreducible closed subsets with $\codim(Z',X)=s$.
Then
\[
Z\neq Z' ~\Leftrightarrow~ \codim(Z,X)\geq s+1.
\]
\end{lemma}

\begin{lemma}\label{lemma:codimensionneq2}
Let $X$ be a scheme and $Z_1,Z_2\subseteq X$ two closed subsets with both $\codim(Z_1,X)\geq s$ and $\codim(Z_2,X)\geq s$.
Then $\codim(Z_1\cup Z_2, X)\geq s$.
\end{lemma}

\begin{lemma}\label{lemma:codimensionsadd}
Let $X$ be an irreducible catenary scheme and $Z\subseteq Z'\subseteq X$ two irreducible closed subsets.
Then
\[
\codim(Z,X) = \codim(Z, Z') + \codim(Z', X). 
\]
\end{lemma}

For an integer $s\geq 0$, define
\[
 X^{(s)} = \{z\in X \mid \codim(\overline{\{z\}}, X)=s\}
\]
and say that \emph{$z$ has codimension $s$ in $X$} if $z\in X^{(s)}$.
Note, that $\overline{\{z\}}$ is always an irreducible closed subset of $X$ and $z$ is its generic point.
One has $X^{(s)}=\emptyset$ for $s>\dim(X)$ as $\dim(X)=\sup_{x\in X}\dim(\OO_{X,x})$.
We have $\codim(Z,X)=0$ if and only if $Z$ contains a whole irreducible component of $X$.
Hence a point $z\in X$ is a generic point of an irreducible component of $X$ if and only if $\dim(\OO_{X,z})=0$. Thus, $X^{(0)}$ is precisely the set of generic points of irreducible components of $X$.

\begin{remark}
Please note that the inequality
\[
 \dim(Z)+\codim(Z,X)\leq \dim(X)
\]
is not always an equality, even for catenary schemes.
\end{remark}

\section{The coniveau filtration}\label{Section: coniveau filtration}

\begin{sect}
In this section we define the coniveau filtration for a Nisnevich local fibrant spectrum on $\Sm_S$.
We fix a spectrum $E\in\Spt_{S^1}(\Sm_S)$.
\end{sect}

\begin{definition}\label{defi:E/Zdefi}
Let $X/S\in \Sm_S$, $Z\subseteq X$ a closed subset with complementary open immersion $j\colon (X\setminus Z)\hookrightarrow X$.
Let $E_X$ be the spectrum $\delta_{X/S,\ast}E$ in $\Spt_{S^1}(X_\Nis)$.
Lemma \ref{lem: restrictions to small sites vs restrictions in small sites} induces a morphism
\begin{equation}\label{canonicalmapeta}
 \eta_j\colon E_X\to j_*j^* E_X = j_\ast E_{X\setminus Z}
\end{equation}
in $\Spt_{S^1}(X_{\Nis})$ whose homotopy fibre, we denote by $E_{Z/X}\in \Spt_{S^1}(X_{\Nis})$.
\end{definition}

\begin{remark}
Note that $E_{X/X}=E_X$ and $E_{\emptyset/X}=*$ in $\Spt_{S^1}(X_{\Nis})$.
\end{remark}

\begin{remark}
Throughout Sections \ref{Section: coniveau filtration} and \ref{Section: Gersten complex}, we could as well work with arbitrary spectra $E_X \in\Spt_{S^1}(X_\Nis)$, not necessarily of the form $\delta_{X/S,\ast}E$ for a spectrum $E\in\Spt_{S^1}(\Sm_S)$.
In this case, $E_{\tilde{X}}$ will be \emph{defined} as $f^\ast E_X$ for an \'{e}tale morphism $f\colon \tilde{X} \to X$ of finite type.
\end{remark}

\begin{sect}\label{para: situation etale bc}
We fix an \'{e}tale morphism $f\colon \tilde{X}\to X$ of finite type and a closed subset $Z\subseteq X$.
Let $\tilde{Z}=Z\times_X \tilde{X}$ denote the pullback of $Z$ along $f$.
We get an induced morphism $\tilde{f}\colon (\tilde{X}\setminus \tilde{Z})\to (X\setminus Z)$ on the open complements.
By Lemma \ref{lem: restrictions to small sites vs restrictions in small sites}, $f^\ast E_X \cong E_{\tilde{X}}$ and we have a canonical morphism $E_X \to f_\ast E_{\tilde{X}}$.
\end{sect}

\begin{lemma}\label{easyidentification}
In the situation of \ref{para: situation etale bc}, we have $f^\ast E_{Z/X} \simeq E_{\tilde{Z}/\tilde{X}}$ in $\Spt_{S^1}(\tilde{X}_{\Nis})$. 
\end{lemma}
\begin{proof}
Let $j\colon (X\setminus Z)\hookrightarrow X$ and $\tilde{j}\colon (\tilde{X}\setminus \tilde{Z})\hookrightarrow \tilde{X}$.
First, we apply the homotopy exact functor $f^*$ to the homotopy fibre sequence $E_{Z/X}\to E_X\to j_*j^* E_X$.
Base change (since $f\in X_\Nis$) for the pullback square
\[
\begin{tikzcd}
\tilde{X}\setminus \tilde{Z} \arrow[r, "\tilde{j}"] \arrow[d, "\tilde f"'] & \tilde{X}\arrow[d, "f"]\\
X\setminus Z   \arrow[r, "j"]                  & X
\end{tikzcd} 
\]
yields an equivalence $f^* j_* \simeq \tilde{j}_* \tilde{f}^*$ and it remains to observe that $\tilde{f}^* j^* \simeq (j \tilde {f})^*\simeq (f \tilde{j})^*\simeq 
\tilde{j}^* f^*$, which is easy.
\end{proof}

\begin{lemma}[Forget support map]\label{lemma:forgetsupport}
Let $Z$ and $Z'$ be closed subsets of $X$ with $Z\subseteq Z'\subseteq X$ and let $E_X\in\Spt_{S^1}(X_{\Nis})$ be a spectrum.
Then there is a canonical \emph{forget support map} $E_{Z/X}\to E_{Z'/X}$.
Further, this map sits in a canonical exact triangle
\[
E_{Z/X}\to E_{Z'/X}\to j_* E_{(Z'\setminus Z)/(X\setminus Z)}
\]
of objects from $\Spt_{S^1}(X_{\Nis})$, where $j\colon (X\setminus Z)\hookrightarrow X$.
\end{lemma}
\begin{proof}
Note first that $j'\colon  (X\setminus Z')\hookrightarrow X$ factorizes as
\[
X\setminus Z' = (X\setminus Z)\setminus (Z'\setminus Z) \xrightarrow{~k~} X\setminus Z \xrightarrow{~j~} X,
\]
and therefore the unit $\id\to k_*k^*$ of the adjunction $k^* \dashv k_*$ induces a morphism
$j_*j^*\to j_*k_*k^*j^*\simeq j'_*j'^*$.
This map is compatible with the units of the adjunctions $j'^* \dashv j'_*$ and $j^* \dashv j_*$, thus inducing the forget support map $E_{Z/X}\to E_{Z'/X}$ on homotopy fibres.

For the exact triangle consider the diagram
\[
\begin{tikzcd}
F[-1]    \arrow[r] \arrow[d]    &   {*} \arrow[r] \arrow[d]  		& F \arrow[d]\\
E_{Z/X} \arrow[r] \arrow[d, "\text{forget support}"'] 	&   E_X	\arrow[r] \arrow[d, "\id"]		& j_* j^* E_X \arrow[d]\\
E_{Z'/X}  \arrow[r]		&   E_X	\arrow[r]			& j'_*j'^*E_X
\end{tikzcd} 
\]
of distinguished triangles.
It remains to identify $F$ with $j_* E_{(Z'\setminus Z)/(X\setminus Z)}$.
Applying the exact functor $j_*$ to the exact triangle
\[
E_{(Z'\setminus Z)/(X\setminus Z)} \to    E_{(X\setminus Z)} \to k_*k^* E_{(X\setminus Z)},
\]
yields the right vertical exact triangle
\[
F=j_* E_{(Z'\setminus Z)/(X\setminus Z)} \to j_* j^* E_X \to j_* k_*k^*j^* E_X,
\]
where we used the definition $j'^* E_X \simeq E_{(X\setminus Z)}$.
\end{proof}

\begin{remark}
In particular, the previous Lemma~\ref{lemma:forgetsupport} induces in particular a long exact sequence
\[
\cdots\to E^n_{Z}(X) \to E^n_{Z'}(X)\to E^n_{Z'\setminus Z}(X\setminus Z)\to E^{n+1}_{Z}(X)\to\cdots.
\]
\end{remark}

\begin{sect}\label{para: Nisnevich squares}
Recall that a Nisnevich distinguished square
\begin{equation*}
\begin{tikzcd}
\tilde U \arrow[r] \arrow[d] & \tilde X \arrow[d, "f"]\\
U        \arrow[r, "j"] & X
\end{tikzcd} 
\end{equation*}
is a pullback square such that $j$ is an open immersion, $f$ is an \'etale morphism of finite type and $(X- j(U))_\mathrm{red} \times_X\tilde{X}\to (X- j(U))_\mathrm{red}$ is an isomorphism.
\end{sect}

\begin{sect}
Recall moreover, that an object-wise fibrant spectrum $E\in\Spt_{S^1}(\Sm_S)$ (\ie, every evaluation $E(X)$ is an ordinary omega spectrum) is \emph{Nisnevich local fibrant} if and only if $E(\emptyset)=*$ and for each Nisnevich distinguished square $Q$ as in \ref{para: Nisnevich squares}, the square $E(Q)$ is a homotopy pullback square (or equivalently a homotopy pushout square).
Equivalently, the sequence
\[
E(X)\to E(\tilde{X})\oplus E(U)\to E(\tilde{U}) 
\]
is a distinguished triangle and hence induces long exact sequences on homotopy groups.
The same observation holds for the Nisnevich local fibrant objects of $\Spt_{S^1}(X_\Nis)$ and the right adjoint $\delta_{X/S,*}\colon \Spt_{S^1}(\Sm_S) \to \Spt_{S^1}(X_{\Nis})$ of the adjunction \eqref{restrictionadjunction} preserves Nisnevich local fibrant objects.
\end{sect}

\begin{lemma}\label{lemma:NisnevichFibrancyConditionReformulated}
An object-wise fibrant spectrum $E\in \Spt_{S^1}(\Sm_S)$ is Nisnevich local fibrant if and only if for all Nisnevich distinguished squares as in \ref{para: Nisnevich squares}, the induced morphism
\[
 E_{Z/X}\to f_*f^* E_{Z/X} \simeq f_* E_{\tilde Z/\tilde X}
\]
(see Lemma~\ref{easyidentification}) is an equivalence.
Here, $Z:=X\setminus U$ and $\tilde Z:=(\tilde X\setminus \tilde U)\cong f^{-1}(Z)$.
\end{lemma}
\begin{proof}
This follows immediately from the fact the a square of spectra is a homotopy pullback square if and only if the homotopy fibres of the horizontal morphisms are equivalent. 
\end{proof}

\begin{lemma}
Let $E\in\Spt_{S^1}(\Sm_S)$ be a Nisnevich local fibrant spectrum.
Let $Z_1$ and $Z_2$ be closed subsets of $X$, and set $Z_{12} = Z_1\cap Z_2$ and $Z = Z_1 \cup Z_2$.
Then the forget support maps
\begin{equation}\label{homotopyfibresquareforNis}
\begin{tikzcd}
E_{Z_{12}/X} \arrow[r] \arrow[d] & E_{Z_2/X} \arrow[d]\\
E_{Z_1/X}        \arrow[r] & E_{Z/X}
\end{tikzcd} 
\end{equation}
form a homotopy (co)fibre square in $E\in\Spt_{S^1}(X_\Nis)$.
\end{lemma}
\begin{proof}
First, observe that the respective open immersions form a Nisnevich distinguished square
\[
\begin{tikzcd}
X\setminus Z \arrow[r] \arrow[d] & X\setminus Z_2  \arrow[d]\\
X\setminus Z_1        \arrow[r] & X\setminus {Z_{12}} 
\end{tikzcd} 
\]
Denote by $j\colon X\setminus Z\hookrightarrow X$ and $j_i\colon X\setminus Z_i\hookrightarrow X$ for $i=1$, $i=2$ or $i=12$ the complementary open immersions.
Since $E$ is Nisnevich local fibrant, it follows that 
\[
\begin{tikzcd}
j_{12,*}j_{12}^* E_{X} \arrow[r] \arrow[d] & j_{2,*}j_{2}^* E_{X} \arrow[d]\\
j_{1,*}j_{1}^* E_{X}        \arrow[r] & j_*j^* E_{X}
\end{tikzcd} 
\]
is a homotopy pullback square.
Mapping into this square from the square with edges ${\rm id}_{E_X}$, which is a homotopy pullback square for trivial reasons, and taking homotopy fibres, yields the square \eqref{homotopyfibresquareforNis}.
Thus, \eqref{homotopyfibresquareforNis} is a homotopy pullback square, too.
\end{proof}

\begin{corollary}\label{corollary:disjointunions}
Let $E\in\Spt_{S^1}(\Sm_S)$ be a Nisnevich local fibrant spectrum and $Z_1,\ldots, Z_r\subseteq X$ be disjoint closed subsets.
Then
\[
 \bigoplus_{i=1}^r E_{Z_i/X} \simeq E_{(\coprod_{i=1}^r Z_i/X)}.
\]
\end{corollary}

\begin{definition}\label{defi:filtrationpresheaf}
Let $E_X\in\Spt_{S^1}(X_\Nis)$ be a spectrum.
For an integer $s\geq 0$, we define the spectrum
\[
E_{X^{(s)}}:=\hspace{-5pt}\underset{ \substack{ Z\subseteq X \text{closed} \\ \codim(Z,X)\geq s }}{\colim} E_{Z/X}
\]
in $\Spt_{S^1}(X_\Nis)$.
The structure maps for the colimit are the forget support maps (see Lemma~\ref{lemma:forgetsupport}).
\end{definition}

\begin{remark}
Informally, one should think of the colimit in the previous Definition~\ref{defi:filtrationpresheaf} as ``making the $Z$'s bigger''. The index category is filtered as one can take the union of two closed sets.
\end{remark}

\begin{sect}\label{para: filtration}
Since a closed subset of $X$ of codimension $\geq (s+1)$ is in particular a closed subset of codimension $\geq s$, we get a filtration
\begin{equation*}
* \to E_{X^{(d)}} \to \cdots \to E_{X^{(s+1)}} \to  E_{X^{(s)}} \to\cdots \to E_{X^{(0)}} \cong E_X
\end{equation*}
of presheaves of spectra on $X_\Nis$.
For the last equivalence, observe that the colimit in Definition~\ref{defi:filtrationpresheaf} has a terminal object $Z=X$ in the case $s=0$.
\end{sect}

\begin{definition}\label{defi:cofibre}
We denote the homotopy cofibre of $E_{X^{(s+1)}} \to  E_{X^{(s)}}$ by $E_{X^{(s/s+1)}}$.
\end{definition}

\begin{sect}
As usual, one can associate a spectral sequence (more precisely, a presheaf on $X_\Nis$ of spectral sequences) to such a situation:
Applying $\pi_n$ for an integer $n$ to the filtration of \ref{para: filtration} yields a finite filtration
\[
0 \subseteq \im\pi_n(E_{X^{(d)}} \to E_X) \subseteq \ldots \subseteq  \im\pi_n(E_{X^{(1)}} \to E_X) \subseteq \pi_n(E_X)
\]
and the associated spectral sequence
\[
 E^1_{p,q} =\pi_{p+q}(E_{X^{(p/p+1)}}) \Rightarrow \pi_{p+q}(E_X)
\]
is degenerate (and hence always converges in the strongest sense) as the filtration above is bounded.
Reindexing and rephrasing along Definition~\ref{defi:homotopygroups}, we get
\[
 E_1^{p,q} =E^{p+q}_{X^{(p/p+1)}} \Rightarrow E^{p+q}.
\]
The constructed spectral sequence is not yet the coniveau spectral sequence.
To obtain the latter, we will sheafify the whole situation (after taking homotopy groups as above) and identify the homotopy cofibres $E_{X^{(p/p+1)}}$ with certain coproducts.
\end{sect}

\begin{proposition}\label{proposition:EvalueationOfCofibres}
Let $E\in\Spt_{S^1}(\Sm_S)$ be a Nisnevich local fibrant spectrum.
Then for every integer $s\geq 0$, we have an equivalence
\[
 E_{X^{(s/s+1)}} \simeq \bigoplus_{z\in X^{(s)}} \mathfrak{z}_*\mathfrak{z}^*  E_{Z/X},
\]
where $\mathfrak{z}^*: \Spt_{S^1}(X_\Nis)\rightleftarrows \Spt_{S^1}(\Spec(\OO_{X,z})_\Nis): \mathfrak{z}_*$ is the adjunction \eqref{adjunctiononthesmallsites} for the canonical morphism $\mathfrak{z}\colon \Spec(\OO_{X,z})\to X$ and where we set $Z:=\overline{\{z\}}$ in each summand by abuse of notation.
\end{proposition}
\begin{proof}
For closed subsets $Z$ and $Z'$ of $X$ with $Z\subseteq Z'$, Lemma~\ref{lemma:forgetsupport} yields an exact triangle
\[
 E_{Z/X} \to E_{Z'/X} \to j_* E_{(Z'\setminus Z)/(X\setminus Z)}.
\]
Taking filtered colimits yields an exact triangle
\begin{equation}\label{equation:otherhocofibredescription}
 E_{X^{(s+1)}} \to  E_{X^{(s)}} \to \hspace{-5pt}\underset{ \substack{ Z,Z'\subseteq X \text{closed} \\ Z\subseteq Z' \\ \codim(Z',X)\geq s\\ \codim(Z,X)\geq s+1 }}{\colim} j_* E_{(Z'\setminus Z)/(X\setminus Z)}
\end{equation}
of objects from $\Spt_{S^1}(X_\Nis)$.
In particular, the right-hand side is equivalent to $E_{X^{(s/s+1)}}$.
The colimit in \eqref{equation:otherhocofibredescription} runs over the filtered category of pairs $(Z,Z')$ where $Z\subseteq Z'$ for $Z,Z'\subseteq X$ closed subsets of the indicated codimensions with an arrow $(Z,Z')\to (\hat Z, \hat Z')$ if and only if both $Z\subseteq \hat Z$ and $Z'\subseteq \hat Z'$.
\\
We will now rewrite this colimit.
Fix a pair $(Z,Z')$.
Since $X$ is noetherian, $Z'$ is noetherian as a topological space and hence the union of its finite number of irreducible components $Z'_1,\ldots,Z'_r$, each of codimension $\geq s$.
It follows, that all the intersections $Z'_i \cap Z'_j$ for $i\neq j$ are of codimension $\geq s +1$ by Lemma~\ref{lemma:codimensionneq}.
Set
\begin{equation*}
 \hat Z := Z \cup \bigcup_{i\neq j} (Z'_i \cap Z'_j) \cup \hspace{-5pt}\underset{\substack{i ~\text{such that}\\ \codim(Z'_i,X)\geq s+1}}\bigcup Z'_i.
\end{equation*}
By Lemma~\ref{lemma:codimensionneq2}, $\hat Z$ has codimension $\geq s+1$ and $(\hat Z, \hat Z':=Z')$ receives a map from our original pair $(Z,Z')$.
Let $T\subseteq X^{(s)}$ be the set of generic points of those $Z'_i$ of codimension $s$.
Then $\overline{T} \cup \hat Z=\hat Z'$.
Further, $\hat U = X \setminus \hat Z \subseteq X$ is an open \emph{separating} neighbourhood of $T$.
By this we mean that $\overline{T} \cap \hat U$ splits into a disjoint union of closures of points of $T$ in $\hat U$.
Combining these observations, we get a cofinal functor from the category of pairs $(T,U)$ with $T \subseteq X^{(s)}$ a finite subset and $U \subseteq X$ an open separating neighbourhood of $T$ into our original index category by mapping a pair $(T,U)$ to the pair $(X\setminus U,\overline{T}\cup (X\setminus U))$.
In particular,
\begin{equation*}
 \underset{ \substack{ Z,Z'\subseteq X \text{closed} \\ Z\subseteq Z' \\ \codim(Z',X)\geq s\\ \codim(Z,X)\geq s+1 }}{\colim} j_* E_{(Z'\setminus Z)/(X\setminus Z)}
 \simeq
 \hspace{-5pt}\underset{ \substack{T \subseteq X^{(s)} ~\text{finite} \\ T\subseteq U\subseteq X ~\text{open sep nbh}}}{\colim} j_* E_{(\overline{T}\cap U)/U}.\hspace{7.5ex}
\end{equation*}
As $U$ is a separating neighbourhood of $T$, Corollary \ref{corollary:disjointunions} gives a splitting
\begin{equation*}
 E_{(\overline{T}\cap U)/U} \simeq \bigoplus_{z\in T} E_{(\overline{ \{z\} }\cap U)/U}.
\end{equation*}
Note that the open separating neighbourhoods of $T$ are cofinal in all open neighbourhoods of $T$.
In particular, we get
\begin{equation*}
\begin{array}{rcl}
\hspace{1.3ex}
\underset{ \substack{T \subseteq X^{(s)} ~\text{finite} \\ T\subseteq U\subseteq X ~\text{open sep nbh}}}{\colim} j_* E_{(\overline{T}\cap U)/U} \!\!&\!\simeq \!&\! \underset{ \substack{T \subseteq X^{(s)} ~\text{finite} \\ T\subseteq U\subseteq X ~\text{open nbh}}}{\colim} \bigoplus_{z\in T} j_\ast E_{(\overline{ \{z\} }\cap U)/U}\vspace{2ex}
\\
                \!\!&\!\simeq \!&\! \bigoplus_{z\in X^{(s)}} \underset{ z\in  U\subseteq X ~\text{open nbh}}{\colim} j_\ast E_{(\overline{ \{z\} }\cap U)/U}
\end{array}
\end{equation*}
Finally, by Lemma \ref{easyidentification}, $E_{(\overline{ \{z\} }\cap U)/U} \simeq j^\ast E_{\overline{ \{z\} } / X}$ and the claim follows.
\end{proof}

\begin{corollary}\label{corollary:howthecofibreslook}
Let $E\in\Spt_{S^1}(\Sm_S)$ be a Nisnevich local fibrant spectrum.
Then for every integer $s\geq 0$ and every integer $q$, we have an isomorphism
\[
 E^q_{X^{(s/s+1)}} \cong \bigoplus_{z\in X^{(s)}} \mathfrak{z}_*\mathfrak{z}^*  E^q_{Z/X}.
\] 
\end{corollary}
\begin{proof}
By Proposition~\ref{proposition:EvalueationOfCofibres} $E_{X^{(s/s+1)}} \simeq \bigoplus_{z\in X^{(s)}} \mathfrak{z}_*\mathfrak{z}^*  E_{Z/X}$.
Using Lemma \ref{basechangeonhomotopygroups} we compute
\[
\begin{array}{rcl}
\pi_{-q}(\bigoplus_{z\in X^{(s)}} \mathfrak{z}_*\mathfrak{z}^*  E_{Z/X}) & \cong & \bigoplus_{z\in X^{(s)}}\pi_{-q}( \mathfrak{z}_*\mathfrak{z}^*  E_{Z/X})\\
								      & \cong & \bigoplus_{z\in X^{(s)}} \mathfrak{z}_*\mathfrak{z}^* \pi_{-q}(E_{Z/X}). 
\end{array}\qedhere
\] 
\end{proof}

\begin{sect}
Recall that a sheaf $F$ of abelian groups on the site $X_\Nis$ is called \emph{flabby}, if the presheaf $\mathrm{H}^q(-,F)$ on $X_\Nis$ is zero for $q\neq 0$.
A flabby sheaf is in particular acyclic, \ie, $\mathrm{H}^q(X,F)=0$ for $q\neq 0$.
\end{sect}

\begin{proposition}\label{proposition:flasque}
Let $E\in\Spt_{S^1}(\Sm_S)$ be a Nisnevich local fibrant spectrum, $z \in X^{(s)}$ with $Z:=\overline{\{z\}}$ and $q$ an integer.
Then the presheaf
\[
\mathfrak{z}_*\mathfrak{z}^* E_{Z/X}^{q}
\]
of abelian groups is a flabby sheaf on $X_\Nis$.
\end{proposition}

\begin{proof}
Let $q$ be an integer.
Let $V \rightarrow X$ be \'{e}tale of finite type with (set-theoretical) fibre $V(z)$ over the point $z\in X$.
For a point $v\in V(z)$, we set $V_{v}^{\rm loc} := {\rm Spec}(\mathcal{O}_{V,v})$ and $V_{z}^{\rm loc}:=V \times_XX_{z}^{\rm loc}$.
Using the identification $Z\times_{X}V_{z}^{\rm loc} = \coprod_{v\in V(z)} v$, we have
\[
\begin{array}{rcll}
\mathfrak{z}_*\mathfrak{z}^* E_{Z/X}^q(V)
&\cong&	\pi_{-q}\left( \mathfrak{z}_*\mathfrak{z}^* E_{Z/X}(V) \right)	&\text{(by Lemma~\ref{basechangeonhomotopygroups})}\\
&\cong& \pi_{-q}\left( E_{Z\times_{X}V_{z}^{\rm loc} / V_z^{\rm loc}}(V_z^{\rm loc}) \right) &\text{(by Lemma~\ref{easyidentification})}\\  
&\cong& \pi_{-q}\left( E_{\left(\coprod_{v\in V(z)} v\right) / V_z^{\rm loc}}(V_z^{\rm loc}) \right) &\\  
&\cong& \pi_{-q}\left( \bigoplus_{v\in V(z)} E_{v/V_{z}^{\rm loc}}(V_{z}^{\rm loc}) \right) &\text{(by Corollary~\ref{corollary:disjointunions})}\\  
&\cong& \bigoplus_{v\in V(z)} E_{v/V_{v}^{\rm loc}}^q(V_{v}^{\rm loc}).&
\end{array}
\]
Let us now prove the sheaf property of $\mathfrak{z}_*\mathfrak{z}^* E_{Z/X}^{q}$.
Using Lemma~\ref{easyidentification}, we may restrict us to Nisnevich covers $V\rightarrow X$ of $X$.
Writing $W := V\times_XV$, we have to show that
\[
\begin{tikzcd}[column sep=small]
0 \arrow[r]&
\mathfrak{z}_*\mathfrak{z}^*E_{Z/X}^q(X) \arrow[d, equal]\arrow[r]&
\mathfrak{z}_*\mathfrak{z}^*E_{Z/X}^q(V) \arrow[d, equal]\arrow[r,"{\rm pr}_1^* - {\rm pr}_2^*"]&
\mathfrak{z}_*\mathfrak{z}^*E_{Z/X}^q(W) \arrow[d, equal]
\\
&
E_{z/X_z^{\rm loc}}^q(X_z^{\rm loc})\arrow[r] &
\bigoplus_{v\in V(z)} E_{v/V_{v}^{\rm loc}}^q(V_{v}^{\rm loc})\arrow[r] &
\bigoplus_{w\in W(z)} E_{w/W_{w}^{\rm loc}}^q(W_{w}^{\rm loc})
\end{tikzcd}
\]
is an exact sequence.
Since $V\to X$ is a Nisnevich cover, we can find a point $v_0$ in the fibre $V(z)$ with residue field $k(v_0) = k(z)$.
In particular, by Lemma~\ref{lemma:NisnevichFibrancyConditionReformulated}, the composition
\[
  E_{z/X_z^{\rm loc}}^q(X_z^{\rm loc}) \longrightarrow  \bigoplus_{v\in V(z)} E_{v/V_{v}^{\rm loc}}^q(V_{v}^{\rm loc}) \xrightarrow{\rm can} E_{v_0/V_{v_0}^{\rm loc}}^q(V_{v_0}^{\rm loc})
\]
is an isomorphism, which settles the exactness at $\mathfrak{z}_*\mathfrak{z}^*E_{Z/X}^q(X)$.\\
For the exactness at $\mathfrak{z}_*\mathfrak{z}^*E_{Z/X}^q(V)$, observe that $\bigoplus_{v\in V(z)} E_{v/V_{v}^{\rm loc}}^q(V_{v}^{\rm loc})$ splits into a direct sum of $E_{z/X_z^{\rm loc}}^q(X_z^{\rm loc})$ and $\bigoplus_{v_0 \neq v\in V(z)} E_{v/V_{v}^{\rm loc}}^q(V_{v}^{\rm loc})$.
Hence, it is enough to show that the restricted map $\bigoplus_{v_0 \neq v\in V(z)} E_{v/V_{v}^{\rm loc}}^q(V_{v}^{\rm loc}) \rightarrow \bigoplus_{w\in W(z)} E_{w/W_{w}^{\rm loc}}^q(W_{w}^{\rm loc})$ is a mono{\-}morphism.
To this end, it is suffices to show that
\[
 {\rm pr}_1^* - {\rm pr}_2^* \colon 
  E_{v/V_{v}^{\rm loc}}^q(V_{v}^{\rm loc}) \longrightarrow
  E_{v\otimes v_0/W_{v\otimes v_0}^{\rm loc}}^q(W_{v\otimes v_0}^{\rm loc}) \oplus E_{v_0\otimes v/W_{v_0\otimes v}^{\rm loc}}^q(W_{v_0\otimes v}^{\rm loc})
\]
is a monomorphism for each $v$ different from $v_0$.
But even the projection
\[
{\rm pr}_1^* \colon E_{v/V_{v}^{\rm loc}}^q(V_{v}^{\rm loc}) \longrightarrow E_{v\otimes v_0/W_{v\otimes v_0}^{\rm loc}}^q(W_{v\otimes v_0}^{\rm loc})
\]
is an isomorphism by Lemma~\ref{lemma:NisnevichFibrancyConditionReformulated}:
Indeed, the equality $k(v\otimes v_0) = k(v)$ follows from $k(v_0) = k(z)$, so ${\rm pr}_1 \colon (W_{v\otimes v_0}^{\rm loc}, v\otimes v_0) \rightarrow (V_{v}^{\rm loc},v)$ is (essentially) a Nisnevich neighbourhood.
This finishes the proof of the sheaf property of $\mathfrak{z}_*\mathfrak{z}^* E_{Z/X}^{q}$.
\\
In order to show the flabbieness, let us first show that $\mathfrak{z}^* E_{Z/X}^{q}$ is flabby:
Again, we have $\mathfrak{z}^* E_{Z/X}^{q} = E_{z/X_{z}^{\rm loc}}^q$ by Lemma~\ref{easyidentification}.
Let $j\colon U \hookrightarrow X_{z}^{\rm loc}$ be the open complement of the closed point $z \in X_{z}^{\rm loc}$.
Then $j^*E_{z/X_{z}^{\rm loc}}^q$ is trivial by construction.
Hence, $E_{z/X_{z}^{\rm loc}}^q$ is supported on $z$, \ie, $E_{z/X_{z}^{\rm loc}}^q = z_*z^*E_{z/X_{z}^{\rm loc}}^q$ is flabby as a skyscraper-sheaf.
For the flabbiness of $\mathfrak{z}_*\mathfrak{z}^* E_{Z/X}^{q}$, we have to show that ${\rm H}^i(V_{\rm Nis}, \mathfrak{z}_*\mathfrak{z}^* E_{Z/X}^{q})$ is trivial for all $V \rightarrow X$ \'{e}tale of finite type and $i>0$.
Since $\mathfrak{z}^* E_{Z/X}^{q}$ is flabby, it is $\mathbb{R}\mathfrak{z}_*$-acyclic, \ie, $\mathfrak{z}_*\mathfrak{z}^* E_{Z/X}^{q} \simeq \mathbb{R}\mathfrak{z}_*\mathfrak{z}^* E_{Z/X}^{q}$.
In particular, we have
\[
 {\rm H}^i(V_{\rm Nis}, \mathfrak{z}_*\mathfrak{z}^* E_{Z/X}^{q}) \cong
 {\rm H}^i(V_{\rm Nis}, \mathbb{R}\mathfrak{z}_*\mathfrak{z}^* E_{Z/X}^{q}) \cong
 {\rm H}^i(V_{z,{\rm Nis}}^{\rm loc}, \mathfrak{z}^* E_{Z/X}^{q}),
\]
but the latter group is trivial since $V_{z,{\rm Nis}}^{\rm loc} \rightarrow X_{z}^{\rm loc}$ is \'{e}tale and $\mathfrak{z}^* E_{Z/X}^{q}$ is flabby.
\end{proof}

\begin{corollary}\label{corollary:sumisflabby}
Let $E\in\Spt_{S^1}(\Sm_S)$ be a Nisnevich local fibrant spectrum.
Then for every integer $s\geq 0$ and every integer $q$, the presheaf $E_{X^{(s/s+1)}}^q$ is a flabby sheaf on $X_\Nis$.
\end{corollary}
\begin{proof}
By Corollary \ref{corollary:howthecofibreslook}, we have $E^q_{X^{(s/s+1)}} \simeq \bigoplus_{z\in X^{(s)}} \mathfrak{z}_*\mathfrak{z}^*  E^q_{Z/X}$.
Here the direct sum is the direct sum of presheaves.
But $X_\Nis$ in noetherian (see e.g.~\cite[Proposition 5.2]{Voevodsky07}), so the direct sum of sheaves is the direct sum of presheaves and the claim follows from Proposition \ref{proposition:flasque}.
\end{proof}

\section{The Nisnevich Gersten complex}\label{Section: Gersten complex}

\begin{lemma}\label{lemma:precousincomplex}
Let $E\in\Spt_{S^1}(\Sm_S)$ be a spectrum and $n$ an integer.
The cofibre sequences of Definition~\ref{defi:cofibre} yield a complex of presheaves on $X_\Nis$ of abelian groups
\[
0\to E_X^{n} \xrightarrow{e} E^{{n}}_{X^{(0/1)}} \xrightarrow{d^{0}} E^{{n+1}}_{{X^{(1/2)}}} \xrightarrow{d^{1}} \cdots \xrightarrow{d^{d-2}} E^{{n+d-1}}_{{X^{(d-1/d)}}} \xrightarrow{d^{d-1}} E^{n+d}_{X^{(d)}} \to 0.
\]
The (Nisnevich) sheafification $(-)^\sim$ of this complex is exact at the first spot $(E_X^{n})^\sim$ if and only if the canonical map
           \begin{equation}\label{eqn:firstexactness}
           \begin{array}{lcccl}
           (\alpha_0^n)^\sim \colon& (E^{n}_{X^{(1)}})^\sim &\to& (E^{n}_{X^{(0)}})^\sim&
           \end{array}
           \end{equation}
           is zero and it is exact at the spot $(E^{{n+s}}_{X^{(s/s+1)}})^\sim$ for $s\geq 0$ if both the canonical maps
           \begin{equation}\label{eqn:secondexactness}
           \begin{array}{rcccl}
            (\alpha_{s-1}^{n+s})^\sim   \colon&(E^{n+s}_{X^{(s)}})^\sim &\to    & (E^{n+s}_{X^{(s-1)}})^\sim &\text{ and}\vspace{1ex}\\
            (\alpha_{s+1}^{n+s+1})^\sim \colon &(E^{n+s+1}_{X^{(s+2)}})^\sim &\to& (E^{n+s+1}_{X^{(s+1)}})^\sim&
           \end{array}
           \end{equation}
           are zero (where the first condition is empty for $s=0$).
\end{lemma}
\begin{proof}
The long exact sequences on homotopy groups associated to the cofibre sequences $E_{X^{(s+1)}}\to E_{X^{(s)}}\to E_{X^{(s/s+1)}}$ from Definition~\ref{defi:cofibre} for $s\geq 0$ yield a diagram
\[
\begin{tikzcd}[column sep=1em]
  & E^n_{X^{(1)}}\arrow[dr, "\alpha_0^n"] &               & && E^{n+2}_{X^{(3)}}\arrow[dr, "\alpha_2^{n+2}"] &&\\
  &               & E^n_{X^{(0)}}\arrow[dr, "\beta_0^n"] & &&& E^{n+2}_{X^{(2)}}\arrow[dr, "\beta_2^{n+2}"] &\\
0\arrow[r] & E^n_X\arrow[ur, equal] \arrow[rr, "e"] && E^n_{X^{(0/1)}}\arrow[dr, "\gamma_0^n"]\arrow[rr, "d^0"] && E^{n+1}_{X^{(1/2)}}\arrow[ur, "\gamma_1^{n+1}"] \arrow[rr, "d^1"] && E^{n+2}_{X^{(2/3)}}\\
  &                       &&& E^{n+1}_{X^{(1)}}\arrow[dr, "\alpha_0^{n+1}"]\arrow[ur, "\beta_1^{n+1}"] &&&&\\
  &                       && E^{n+1}_{X^{(2)}}\arrow[ur, "\alpha_1^{n+1}"]  && E^{n+1}_{X^{(0)}}   &&&
\end{tikzcd}
\]
and we define the middle horizontal sequence as indicated.
This sequence is clearly a complex as the diagonal lines are complexes.
The remaining statement follows immediately from sheafification $(-)^\sim$ applied to the whole diagram.
\end{proof}

\begin{sect}\label{sec:augmentedgerstencomplex}
For a Nisnevich local fibrant spectrum $E\in\Spt_{S^1}(\Sm_S)$ we can rewrite the complex of presheaves on $X_\Nis$ of abelian groups from the previous Lemma~\ref{lemma:precousincomplex} with the help of Corollary~\ref{corollary:howthecofibreslook} as
\begin{multline}\label{eqn:precousincomplex}
0\to E_X^{n} \xrightarrow{e} \bigoplus_{z\in X^{(0)}} \mathfrak{z}_*\mathfrak{z}^*  E_{Z/X}^n \xrightarrow{d^{0}} \bigoplus_{z\in X^{(1)}} \mathfrak{z}_*\mathfrak{z}^*  E_{Z/X}^{n+1} \xrightarrow{d^{1}} \cdots \\
\cdots \xrightarrow{d^{d-2}} \bigoplus_{z\in X^{(d-1)}} \mathfrak{z}_*\mathfrak{z}^*  E^{n+d-1}_{Z/X} \xrightarrow{d^{d-1}} \bigoplus_{z\in X^{(d)}} \mathfrak{z}_*\mathfrak{z}^*  E^{n+d}_{Z/X} \to 0 .
\end{multline}
\end{sect}


\begin{definition}\label{defi:gerstencomplex}
For every integer $n$, we define the \emph{Nisnevich Gersten complex} $\mathcal{G}^\bullet(E,n)$ of $E$ and homotopical degree $n$ as the complex with entries
\[
 \mathcal{G}^s(E,n):= \bigoplus_{z\in X^{(s)}} \mathfrak{z}_*\mathfrak{z}^*  E_{Z/X}^{n+s}
\]
for $s\geq 0$ and zero otherwise. The differentials $d^s$ are defined as in \ref{sec:augmentedgerstencomplex}.
\end{definition}

\begin{sect}
For every integer $n$, we can reformulate~\eqref{eqn:precousincomplex} as a map
\[
E_X^{n} \xrightarrow{e} \mathcal{G}^\bullet(E,n)
\]
into a complex of flabby sheaves (see Corollary~\ref{corollary:sumisflabby}) of abelian groups.
\end{sect}

\begin{sect}\label{sect: abuse of stalk notation}
In abuse of notation, we will just write $E(X_x^h)$ for the stalk of $E$ at a point $x$ of $X$.
Here of course $X_x^h$ denotes the Henselian local scheme ${\rm Spec}(\mathcal{O}_{X,x}^h)$. 
\end{sect}

\begin{proposition}\label{proposition:gerstencomplex}
Let $E\in\Spt_{S^1}(\Sm_S)$ be a Nisnevich local fibrant spectrum and $n$ an integer.
There is a complex of sheaves on $X_\Nis$ of abelian groups
\[
0\to (E_X^{n})^\sim \xrightarrow{\tilde e}  \mathcal{G}^0(E,n) \xrightarrow{d^{0}} \mathcal{G}^1(E,n) \xrightarrow{d^{1}} \cdots  \xrightarrow{d^{d-1}} \mathcal{G}^d(E,n) \to 0
\]
where all but the first entry are flabby Nisnevich sheaves. This complex is
\begin{enumerate}
 \item exact at the first spot $(E_X^{n})^\sim$ if and only if, for each point $x$ of $X$ and all $Z\subseteq X$ closed with $\codim(Z,X)\geq 1$, the forget support map
 \[
  E^n_{Z/X}(X_x^h)\to E^n_{X}(X_x^h)
 \]
 is trivial and
 \item exact at $\mathcal{G}^s(E,n)$ for $s\geq 0$ if, Nisnevich-locally on $X$,
 \begin{enumerate}
 \item if $s= 1$, for each point $x$ of $X$ and all $Z\subseteq X$ closed with $\codim(Z,X)$ $\geq 1$, the forget support map
 \[
  E^{n+1}_{Z/X}(X_x^h)\to E^{n+1}_{X}(X_x^h)
 \]
 is trivial and
 \item if $s> 1$, for all $Z\subseteq X$ closed with $\codim(Z,X)\geq s$, there exists $Z\subseteq Z'\subseteq X$ closed with $\codim(Z',X)\geq s-1$ such that the forget support map
 \[
  E^{n+s}_{Z/X}(X)\to E^{n+s}_{Z'/X}(X)
 \]
 is trivial and
 \item for all $s\geq0$ and all $Z\subseteq X$ closed with $\codim(Z,X)\geq s+2$, there exists $Z\subseteq Z'\subseteq X$ closed with $\codim(Z',X)\geq s+1$ such that the forget support map
 \[
  E^{n+s+1}_{Z/X}(X)\to E^{n+s+1}_{Z'/X}(X)
 \]
 is trivial.
 \end{enumerate}
\end{enumerate}
\end{proposition}

\begin{proof}
The complex is obtained by applying the sheafification functor to the complex~\eqref{eqn:precousincomplex}.
By Corollary~\ref{corollary:sumisflabby}, all but the first entry are flabby Nisnevich sheaves.
The exactness conditions are just expanded versions of \eqref{eqn:firstexactness} and \eqref{eqn:secondexactness}.
\end{proof}

\section{Effaceability}\label{section:effaceability}

\begin{sect}\label{sect:hypothesisonS}
In this chapter, let $S$ be the spectrum of a Henselian discrete valuation ring $\mathfrak{o}$ with infinite residue field $\mathbb{F}$ of characteristic $p$ and quotient field $k$.
Note that we make use of this hypothesis only from Lemma~\ref{lemma:effaceability} onwards.
\end{sect}

\begin{construction}\label{construction:absoluteconstruction}
Let $E\in\Spt_{S^1}(\Sm_S)$ be a spectrum and $f\colon \tilde X\to X$ a morphism in $\Sm_S$.
We consider the morphism
\begin{equation}\label{eqn:firsteta}
 \eta_f\colon E_X\to f_*E_{\tilde X} 
\end{equation}
in $\Spt_{S^1}(X_\Nis)$ given on an \'etale morphism $V\to X$ of finite type as the map
\[
E_X(V)=E(V) \to E(V\times_X \tilde X) = E_{\tilde X}(V\times_X \tilde X)
\]
induced by the projection.
This clearly generalizes the construction~\eqref{canonicalmapeta} where the morphism $f$ was assumed to be \'etale.
Indeed, in this case we have $f^*(E_X)\cong E_{\tilde X}$.
\end{construction}

\begin{construction}\label{construction:relativeconstruction}
Next, for $E\in\Spt_{S^1}(\Sm_S)$, a closed subset $Z\subseteq X$ and a pullback diagram
\begin{equation}\label{eqn:diagramminquestion}
\begin{tikzcd}
\tilde X\setminus \tilde Z \arrow[r, hook, "\tilde j"]\arrow[d, "\tilde f"]&\tilde X \arrow[d, "f"]&\\
X\setminus Z \arrow[r, hook, "j"] &X        
\end{tikzcd}
\end{equation}
we want to define a morphism
\begin{equation}\label{eqn:relativeeta}
\eta_f\colon E_{Z/X}\to f_*E_{\tilde Z/\tilde X} 
\end{equation}
that coincides with~\eqref{eqn:firsteta} for $Z=X$.
First note that the commutative diagram~\eqref{eqn:diagramminquestion} induces the base-change morphism
\begin{equation*}\label{eqn:nummereins}
f^*j_*\xrightarrow{\text{(b.c.)}} \tilde j_*\tilde f^*.
\end{equation*}
Further, by adjunction, Construction~\ref{construction:absoluteconstruction} induces a map
\begin{equation*}\label{eqn:nummerzwei}
\tilde f^* j^*E_X\cong \tilde f^* E_{X\setminus Z} \xrightarrow{\eta_{f}^\flat} E_{\tilde X\setminus \tilde Z}  \cong \tilde j^*E_{\tilde X}. 
\end{equation*}
Composition with the unit yields a morphism
\begin{equation}\label{eqn:longunitmap}
j_*j^*E_X \xrightarrow{\eta} f_*f^*j_*j^* E_X \xrightarrow{f_*\text{(b.c.)}j^*} f_*\tilde j_* \tilde f^*j^* E_X
\xrightarrow{f_*\tilde j_*\eta_{f}^\flat } f_*\tilde j_*\tilde j^* E_{\tilde X}
\end{equation}
which is seen to fit into a commutative square
\[
\begin{tikzcd}
E_X \arrow[d, "\eta_f"]\arrow[r, "\eta_j"] & \arrow[d] j_*j^* E_X\\
f_* E_{\tilde X}\arrow[r, "f_*\eta_j"]        & f_*\tilde j_*\tilde j^* E_{\tilde X}.
\end{tikzcd}
\]
inducing the desired map $\eta_f$ by taking horizontal homotopy fibres.
\end{construction}

\begin{lemma}\label{lemma:commutativeetatriangle}
Let
\[
X_2 \xrightarrow{f_2} X_1 \xrightarrow{f_1} X
\]
be two morphisms of noetherian schemes of finite Krull dimension, $Z\subseteq X$ a closed subset and $Z_1:=Z\times_X X_1$, $Z_2:=Z\times_X X_2$ the respective base changes.
Then we have a commutative triangle
\[
\begin{tikzcd}
E_{Z/X}\arrow[r, "\eta_{f_1}"]  \arrow[dr, "\eta_{f_1f_2}"'] &  {f}_{1,*}E_{Z_1/X_1} \arrow[d, "f_{1,*}\eta_{f_2}"] \\
                                                     &  f_{1,*} f_{2,*}  E_{Z_2/X_2}
\end{tikzcd}
\]
in $\Spt_{S^1}(X_{\Nis})$ of the respective morphisms \eqref{eqn:relativeeta}.
\end{lemma}
\begin{proof}
By adjointness it suffices to show the commutativity of the outer square of the diagram
\[
\begin{tikzcd}[column sep=1.5em]
f_2^*f_1^*j_*(E_U) \arrow[d, "\text{(b.c.)}"]\arrow[r, "f_2^*\text{(b.c.)}"] & f_2^*j_{1,*}f_{1|U}^*(E_U)\arrow[dl, "\text{(b.c.)}f_{1|U}^*"] \arrow[rr, "f_2^*j_{1,*}\eta_{f_{1|U}}^\flat"] && f_2^*j_{1,*}E_{U_1} \arrow[r, "\text{(b.c.)}"] & j_{2,*}f_{2|U_1}^*(E_{U_1}) \arrow[d, "j_{2,*}\eta_{f_{2|U_1}}^\flat"']\\
j_{2,*}(f_1f_2)_{|U}^* (E_U) \arrow[rrrr, "j_{2,*}\eta_{(f_1f_2)_{|U}}^\flat"] &&&& j_{2,*}(E_{U_2})
\end{tikzcd}
\]
where $j\colon U\hookrightarrow X$, $j_1\colon U_1\hookrightarrow X_1$ and $j_2\colon U_2\hookrightarrow X_2$ are the respective open complements of $Z$, $Z_1$ and $Z_2$ and where $f_{1|U}\colon U_1\to U$, $f_{2|U_1}$ and $(f_1f_2)_{|U}$ are the respective restrictions.
The triangle on the left-hand side commutes as base change morphisms are compatible with composition and the commutativity of the remaining part is seen easily.
\end{proof}

\begin{sect}
Recall, that a Nisnevich local fibrant spectrum $E\in\Spt_{S^1}(S)$ is an \emph{$\AA^1$-Nis\-ne\-vich local fibrant} spectrum if $E(X)\to E(X\times \AA^1)$ is an equivalence for all $X\in\Sm_S$.
\end{sect}

\begin{lemma}\label{lemma:etaA1equivalence}
Let $E\in\Spt_{S^1}(\Sm_S)$ be an $\AA^1$-Nisnevich local fibrant spectrum.
Let $X\in\Sm_S$ be a scheme, $Z\subseteq X$ a closed subset and $\pi\colon \AA^1_X\to X$ the projection.
Then the canonical map (c.f.~\eqref{eqn:relativeeta})
\[
E_{Z/X} \xrightarrow{\eta_\pi} \pi_* E_{\AA^1_Z/\AA^1_X}
\]
is a weak equivalence.
\end{lemma}
\begin{proof}
By construction of the map in question as a homotopy fibre, it suffices to show that the two maps
\[
E_X \xrightarrow{\eta_\pi} \pi_* E_{\AA^1_X} 
\]
and
\[
j_*j^* E_X\to \pi_*\tilde j_*\tilde j^ * E_{\AA^1_X} 
\]
from~\eqref{eqn:longunitmap} are both object-wise weak equivalences, This can be checked directly by evaluation on an object $V\to X$ of the site $X_\Nis$. 
\end{proof}

\begin{lemma}\label{lemma:twosectionscoincide}
Let $E\in\Spt_{S^1}(\Sm_S)$ be an $\AA^1$-Nisnevich local fibrant spectrum.
Let $X\in\Sm_S$ be a scheme and $s\colon X\hookrightarrow \AA^1_X$ a section of the projection $\pi\colon  \AA^1_X \to X$.
Then there is a commutative diagram
\[
\begin{tikzcd}
E_{Z/X}\arrow[r, "\eta_\pi"]  \arrow[dr, equal] &  \pi_*E_{\AA^1_Z/\AA^1_X} \arrow[d, "\pi_*\eta_s"] \\
                                                     &  \pi_*s_*E_{Z/X}
\end{tikzcd}
\]
of weak equivalences.
In particular, for another section $s'\colon X\hookrightarrow \AA^1_X$ of the projection, the morphisms $\pi_*\eta_s$ and $\pi_*\eta_{s'}$ are equal in the homotopy category.
\end{lemma}
\begin{proof}
This follows from the previous Lemmas~\ref{lemma:commutativeetatriangle} and~\ref{lemma:etaA1equivalence}.
\end{proof}

\begin{lemma}\label{lemma:effaceability}
Let $E\in\Spt_{S^1}(\Sm_S)$ be an $\AA^1$-Nisnevich local fibrant spectrum.
Let $V\in\Sm_S$ and $Z\hookrightarrow \AA^1_V$ a closed subscheme such that $Z\hookrightarrow \AA^1_V\xrightarrow{\pi} V$ is finite.
Let $\bar Z:={\pi(Z)_{\text{red}}}$ be the reduced image.
Then $\codim(\AA^1_{\bar Z},\AA^1_V)= \codim(Z, \AA^1_V)-1$ and the forget support map induces the trivial morphism
\[
 \pi_* E_{Z/\AA^1_V}  \to \pi_*E_{\AA^1_{\bar Z}/\AA^1_V}
\]
in the homotopy category.
\end{lemma}
\begin{proof}
Consider the diagram
\[
\begin{tikzcd}
\AA^1_V \arrow[d, "j"', hook]  \arrow[r, "\pi"] &  V \\
\PP^1_V \arrow[ur, "\bar \pi"']           & 
\end{tikzcd}
\]
where the non-vertical maps are the projections and $j$ is the canonical open immersion.
Let us first prove that the triangle
\begin{equation}\label{eqn:firstcommutativetriangle}
\begin{tikzcd}[column sep=large]
\bar \pi_* E_{\PP^1_{\bar Z}/\PP^1_{V}}\arrow[r, "\bar \pi_* \eta_{\bar s_\infty}"]  \arrow[dr, "\bar\pi_*\eta_j"'] &  \bar \pi_*\bar s_{\infty,*} E_{\bar Z/V}=E_{\bar Z/V} \arrow[d, "\eta_{\pi}"] \\
                                                     & \pi_* E_{\AA^1_{\bar Z}/\AA^1_V}
\end{tikzcd}
\end{equation}
commutes in the homotopy category, where $\bar s_\infty\colon V\hookrightarrow\PP_V^1$ is the section at infinity.
Let $\bar s_0\colon V\xhookrightarrow{s_0}\AA^1_V\xrightarrow{j}\PP_V^1$ denote the zero-section.
Since by Lemma~\ref{lemma:twosectionscoincide} the morphism
\[
 \pi_*\eta_{s_0}\colon \pi_* E_{\AA^1_{\bar Z}/\AA^1_V}\to \pi_*s_{0,*} E_{\bar Z/V} = E_{\bar Z/V}
\]
is a weak equivalence, it suffices to show that the outer triangle of the enlarged diagram
\[
\begin{tikzcd}
\bar \pi_* E_{\PP^1_{\tilde Z}/\PP^1_{V}}\arrow[r, "\bar \pi_* \eta_{\bar s_\infty}"] \arrow[ddr, bend right=20, "\bar \pi_* \eta_{\bar s_0}"']  \arrow[dr, "\bar\pi_*\eta_j"'] &  E_{\bar Z/V} \arrow[d, "\eta_{\pi}"] \\
                                                     & \pi_* E_{\AA^1_{\bar Z}/\AA^1_V} \arrow[d, "\pi_*\eta_{s_0}", "\simeq"'] \\
                                                     & E_{\bar Z/V}
\end{tikzcd}
\]
commutes.
Indeed, the bottom triangle is obtained by applying $\bar\pi_*$ to a commutative triangle considered in Lemma~\ref{lemma:commutativeetatriangle} for $\bar s_0=js_0$.
By the same Lemma~\ref{lemma:commutativeetatriangle} applied to $\id=\pi s_0$, the right vertical composition is the identity.
Hence, it suffices to show that $\bar \pi_*\eta_{\bar s_0}= \bar \pi_*\eta_{\bar s_\infty}$ holds in the homotopy category.\\
Since the sections $\bar s_0$ and $\bar s_\infty \colon V\to\PP^1_V$ both factorize through the open immersion $j'\colon \PP^1_{V}\setminus s_1(V)\hookrightarrow \PP_V^1$ via $s_0'$ and $s_\infty'\colon V\to \PP^1_{V}\setminus s_1(V)$, we have a commutative diagram
\[
\begin{tikzcd}
\bar \pi_* E_{\PP^1_{\bar Z}/\PP^1_{V}} \arrow[rr, "\bar \pi_* \eta_{j'}"]\arrow[dr, "\bar\pi_* \eta_{\bar s_0}"']&& \bar\pi_* j'_* E_{(\PP^1_{\bar Z}\setminus s_1(V))/(\PP^1_{V}\setminus s_1(V))}\arrow[dl, "\pi'_* \eta_{s_0'}"]\\
 & E_{\bar Z/V} &
\end{tikzcd}
\]
(and likewise for $\bar s_\infty$ and $s'_\infty$).
Here, $\pi'\colon \PP^1_{V}\setminus s_1(V)\to V$ is the projection.
As $\PP^1_{V}\setminus s_1(V)\cong \AA^1_V$ and $\PP^1_{\bar Z}\setminus s_1(V)\cong \AA^1_{\bar Z}$, we obtain $\pi'_* \eta_{s_0'}=\pi'_* \eta_{s_\infty'}$ by Lemma~\ref{lemma:twosectionscoincide}, thus $\bar\pi_*\eta_{\bar s_0}=\bar\pi_*\bar \eta_{s_\infty}$.
Summing up, this yields the commutativity of diagram~\eqref{eqn:firstcommutativetriangle}.\\
In order to show that the morphism in question
\[
 \pi_* E_{Z/\AA^1_V} = \bar \pi_* j_* E_{Z/\AA^1_V} \to \bar \pi_* j_* E_{\AA^1_{\bar Z}/\AA^1_V} = \pi_*E_{\AA^1_{\bar Z}/\AA^1_V}
\]
is trivial in the homotopy category, we consider the diagram
\[
\begin{tikzcd}
\bar \pi_* j_* E_{Z/\AA^1_V} \arrow[r] &  \bar \pi_* j_* E_{{\AA^1_{\bar Z}}/\AA^1_V} & & \\
                                     &                                 & E_{\bar Z/V} \arrow[ul]& \\
\bar \pi_* E_{Z/\PP^1_V}  \arrow[uu,"\simeq", "\bar \pi_*\eta_j"']  \arrow[r] & \bar \pi_* E_{\PP^1_{\bar Z}/\PP^1_V} \arrow[uu, "\bar \pi_*\eta_j"']  \arrow[ur, "\bar\pi_*\bar s_\infty"] \arrow[rr,"\bar\pi_*\eta_{j''}"] & & \bar \pi_* j''_* E_{(\PP^1_{\bar Z}\setminus Z)/(\PP_V^1\setminus Z)} \arrow[ul, "\bar \pi_* j''_* \eta_{s''_\infty}"']
\end{tikzcd}
\]
where the middle triangle is~\eqref{eqn:firstcommutativetriangle}. The left horizontal maps are induced by the respective forget support maps.
For the right triangle, we note that $\bar s_\infty\colon V\hookrightarrow \PP_V^1$ factorizes through $j''\colon \PP^1_V\setminus Z \hookrightarrow \PP^1_V$ via $s''_\infty\colon V\hookrightarrow \PP^1_V\setminus Z$.
The commutativity of the square on the left-hand side is clear.
The triangle on the right-hand side commutes again by Lemma~\ref{lemma:commutativeetatriangle}.
We observe that the lower horizontal line is given by $\bar\pi_*$ applied to the exact triangle of Lemma~\ref{lemma:forgetsupport}.
In particular, it is an exact triangle itself and therefore the composition is trivial.
Finally, the left vertical arrow is a weak equivalence by the excision  Lemma~\ref{lemma:NisnevichFibrancyConditionReformulated}.
Hence the morphism $\pi_* E_{Z/\AA^1_V}  \to \pi_*E_{\AA^1_{\bar Z}/\AA^1_V}$ in question is trivial in the homotopy category.\\
For the assertion $\codim(\AA^1_{\bar Z},\AA^1_V)= \codim(Z, \AA^1_V)-1$ we can argue component-wise on $Z$ so we may assume that $Z$ is irreducible.
Further, we can replace $\AA_V^1$ by a base change along a flat morphism $V^\prime \rightarrow V$.
In particular, we may assume that $V$ is a local scheme with closed point $\bar{Z}$.
As $Z$ is finite over $\bar{Z}$, it is just a finite union of points in the curve $\AA_{\bar{Z}}^1$.
Thus, $\codim(Z,\AA^1_{\bar Z})=1$ and the assertion follows by Lemma~\ref{lemma:codimensionsadd}.
\end{proof}

\begin{proposition}\label{prop: effaceability}
Let $E\in\Spt_{S^1}(\Sm_S)$ be an $\AA^1$-Nisnevich local fibrant spectrum.
Let $X\in\Sm_S$, $Z\hookrightarrow X$ be a closed subscheme and $x\in X$ be a point.
If $x$ lies in the special fibre $X_\sigma$, assume that $Z_\sigma$ does not contain any connected components of $X_\sigma$.
Then, Nisnevich-locally on $X$ around $x$, there exists a $V\in\Sm_S$, a smooth relative curve $p\colon X\to V$ with $Z$ finite over $V$ and a closed subscheme $Z'\hookrightarrow X$ containing $Z$ such that $\codim(Z',X)=\codim(Z, X)-1$ and the forget support map induces the trivial morphism
\[
 p_* E_{Z/X}  \to p_*E_{Z'/X}
\]
in the homotopy category.
In particular, $E_{Z/X}(X)\to E_{Z'/X}(X)$ is trivial in this case.
\end{proposition}

\begin{proof}
Possibly after shrinking $X$ Nisnevich-locally around $x$, we find a Nisnevich distinguished square
\begin{equation}\label{eqn: gabber representation}
\begin{tikzcd}
X \setminus Z  \arrow[r] \arrow[d] & X  \arrow[d,"f"]\\
\AA^1_V\setminus f(Z) \arrow[r] & \AA^1_V
\end{tikzcd}
\end{equation}
such that $Z\hookrightarrow X\xrightarrow{f} \AA^1_V\xrightarrow{\pi} V$ is finite by \cite[Theorem 2.1]{SS16}.
Let $p\colon X\xrightarrow{f} \AA^1_V\xrightarrow{\pi} V$ denote the composition and set $\bar Z:= p(Z)_{\text{red}}$ and $Z' := p^{-1}(\bar Z)$.
Since $f$ and $\pi$ and hence the composition $p$ is flat, the assertion about the codimensions holds true.
By the excision Lemma~\ref{lemma:NisnevichFibrancyConditionReformulated}, the upper horizontal morphism of the diagram
\[
\begin{tikzcd}
E_{f(Z)/\AA^1_V}  \arrow[r,"\simeq"] \arrow[d] & f_* E_{Z/X}  \arrow[d]\\
E_{\AA^1_{\bar Z}/\AA^1_V} \arrow[r] & f_* E_{Z'/X}
\end{tikzcd}
\]
is an equivalence, where the vertical maps are the respective forget support maps and $f^{-1}f(Z)= Z$.
Application of $\pi_*$ yields the commutative diagram
\[
\begin{tikzcd}
\pi_*E_{f(Z)/\AA^1_V}  \arrow[r,"\simeq"] \arrow[d] & p_* E_{Z/X}\phantom{.}  \arrow[d]\\
\pi_*E_{\AA^1_{\tilde Z}/\AA^1_V} \arrow[r] & p_* E_{\bar Z/X}.
\end{tikzcd}
\]
The left vertical morphism is trivial by the previous Lemma~\ref{lemma:effaceability}.
Hence the right vertical morphism is trivial which proves the claim.
\end{proof}

\begin{sect}
Denote by $X_{x,\eta}^h$ the generic fibre ${\rm Spec}(\mathcal{O}_{X,x}^h \otimes_{\mathfrak{o}}k)$ of the Henselian local scheme at $x$.
Similar to \ref{sect: abuse of stalk notation}, by $E(X_{x,\eta}^h)$ we mean $\colim_{(W,w)} E(W_\eta)$, where $(W,w)$ runs through the Nisnevich neighbourhoods of $x$ and $W_\eta$ is the generic fibre.
\end{sect}

\begin{corollary}\label{cor: effaceability on generic fibre}
Under the assumptions of Proposition \ref{prop: effaceability}, the forget support map
\begin{equation*}
 E_{Z/X}(X_{x,\eta}^h)\to E_{X}(X_{x,\eta}^h)
\end{equation*}
is trivial.
\end{corollary}

\begin{proof}
By \cite[Theorem 2.1]{SS16}, there is a cofinal family of Nisnevich neighbourhoods $(W,w)$ of $x$ admitting a Nisnevich distinguished square of the form \eqref{eqn: gabber representation} with the additional finiteness assumption.
We even claim that for such neighbourhoods $(W,w)$, the forget support map $E_{Z/X}(W_\eta)\to E_{X}(W_\eta)$ is trivial.
To show this, we may assume $W=X$, \ie, we assume $X$ admits a Nisnevich distinguished square as in \eqref{eqn: gabber representation} with $Z/V$ finite.
On the generic fibres, we still have a distinguished square
\begin{equation*}
\begin{tikzcd}
X_\eta \setminus Z_\eta  \arrow[r] \arrow[d] & X_\eta  \arrow[d,"f_\eta"]\\
\AA^1_{V_\eta}\setminus f_\eta(Z_\eta) \arrow[r] & \AA^1_{V_\eta}
\end{tikzcd}
\end{equation*}
and as pullback, $Z_\eta/V_\eta$ is still finite.
Accordingly, the arguments in the proof of Proposition \ref{prop: effaceability} go through for $Z_\eta \subseteq X_\eta$, as well.
In particular, the forget support map $E_{Z/X}(X_\eta)\to E_{X}(X_\eta)$ is indeed trivial.
\end{proof}

\begin{theorem}\label{thm: Bloch-Ogus abstract nonsense}
Let $S$ be a Dedekind scheme with only infinite residue fields.
Moreover, let $E\in\Spt_{S^1}(\Sm_S)$ be an $\AA^1$-Nisnevich local fibrant spectrum and $X\in\Sm_S$ of dimension $d$.
The complex
\begin{multline*}
0\to (E_X^{n})^\sim \xrightarrow{\tilde e} \bigoplus_{z\in X^{(0)}} \mathfrak{z}_*\mathfrak{z}^*  E_{Z/X}^n \xrightarrow{d^{0}} \bigoplus_{z\in X^{(1)}} \mathfrak{z}_*\mathfrak{z}^*  E_{Z/X}^{n+1} \xrightarrow{d^{1}} \cdots \\
\cdots \xrightarrow{d^{d-2}} \bigoplus_{z\in X^{(d-1)}} \mathfrak{z}_*\mathfrak{z}^*  E^{n+d-1}_{Z/X} \xrightarrow{d^{d-1}} \bigoplus_{z\in X^{(d)}} \mathfrak{z}_*\mathfrak{z}^*  E^{n+d}_{Z/X} \to 0
\end{multline*}
is exact, possible except at the spots $(E_X^{n})^\sim$ and $\bigoplus_{z\in X^{(1)}} \mathfrak{z}_*\mathfrak{z}^*  E_{Z/X}^{n+1}$.
Moreover, if for each point $x$ of $X$ the forget support map for the special fibre
 \[
  E_{X_\sigma/X}(X_x^h)\to E_{X}(X_x^h)
 \]
is trivial,
then it is exact everywhere and thus a resolution of $(E_X^{n})^\sim$ by flabby Nisnevich sheaves.
In this case, we have
\[
H^p(Y, (E_X^{n})^\sim) \cong H^p(\mathcal{G}^\bullet(E,n)(Y)).
\]
for $Y\in X_\Nis$ which vanishes for $p>d$.
\end{theorem}

\begin{proof}
Since exactness is checked stalk-wise and we can compute the stalk at a point $x\in X$ after henselization of the local scheme obtained from $S$ at the image of $x$, we may assume, that $S$ is the spectrum of a Henselian discrete valuation ring with infinite residue field.
Now the first result follows from Proposition~\ref{proposition:gerstencomplex} and Proposition~\ref{prop: effaceability}.
\\
Suppose the forget support maps $E_{X_\sigma/X}(X_x^h)\to E_{X}(X_x^h)$ are trivial for all points $x$.
By our assumtion and Propositions~\ref{proposition:gerstencomplex} and \ref{prop: effaceability}, it is enough so show that the forget support map $E^n_{Z/X}(X_x^h)\to E^n_{X}(X_x^h)$ is trivial for closed subsets $X_\sigma \subsetneq Z \subsetneq X$.
We may replace $X$ by the Henselian local scheme $X_x^h$.
Write $Z = Z_1 \cup Z_2$ with $Z_1 = X_\sigma$ and $X_\sigma \nsubseteq Z_2$.
Let $U = X \setminus Z$ and $U_i = X \setminus Z_i$ be the respective open complements.
Observe that $U_1 = X_\eta$ and $U = U_{2,\eta}$ are just the generic fibres.
Consider the exact triangles
\[
E_{Z_2/X}(U_1) \to E_X(U_1) \to E_X(U)\phantom{.}
\]
and
\[
E_{Z_1/X}(X) \to E_X(X) \to E_X(U_1).
\]
By our assumption, the forget support map in the latter triangle is trivial, so the restriction map $E_X(X) \to E_X(U_1)$ admits a retraction $r_1$.
By Corollary \ref{cor: effaceability on generic fibre}, the forget support map in the former triangle is trivial, so the restriction map $E_X(U_1) \to E_X(U)$ admits a retraction $r_2$.
Set $r := r_1 \circ r_2\colon E_X(U) \to E_X(X)$.
By construction, $r$ is a retraction of the restriction map $E_X(X) \to E_X(U)$.
Thus, using the exact triangle
\[
E_{Z/X}(X) \to E_X(X) \to E_X(U),
\]
we get that the forget support map $E_{Z/X}(X) \to E_X(X)$ is indeed trivial.
\end{proof}

\section{A Bloch-Ogus theorem for \'{e}tale cohomology} 
\label{section:blochogus}
In this section we want to apply Theorem \ref{thm: Bloch-Ogus abstract nonsense} to \'{e}tale cohomology.
Let us first fix the situation:

\begin{sect}\label{para: situation etale Bloch-Ogus}
We are in the situation of~\ref{sect:hypothesisonS}.
For the whole section, we fix an essentially smooth scheme $X / S$, connected and of finite dimension.
Let us denote the structural morphism by $p_X\colon X\rightarrow S$.
We fix a coefficient group $\Lambda := \mathbb{Z}/m$ for an integer $m>0$ prime to $p$.
We work in the derived category $\mathcal{D}_c^b(X_{\rm et},\Lambda)$ of bounded (above and below) complexes all of whose cohomology sheaves are constructible sheaves of $\Lambda$-modules.
By an l.c.c.~complex $K^\bullet$, we mean a complexes $K^\bullet \in \mathcal{D}_c^b(X_{\rm et},\Lambda)$ with locally constant cohomology sheaves ${\rm H}^q(K^\bullet)$ for all $q$.
\end{sect}

\begin{sect}\label{para: Eilenberg-MacLane spectrum}
Let $\varepsilon\colon X_{\rm et} \rightarrow X_{\rm Nis}$ be the canonical morphism of sites.
Note that $\mathbb{R}\Gamma(X_{\rm et},-)  \simeq \mathbb{R}\Gamma(X_{\rm Nis}, \mathbb{R}\varepsilon_\ast (-))$.
By abuse of notation, let us denote by $\varepsilon$ also the corresponding morphism $\Sm_{S,{\rm et}} \rightarrow \Sm_{S,{\rm Nis}}$ of the smooth sites.
For an l.c.c.~complex $K^\bullet$ in $\mathcal{D}_c^b(S_{\rm et},\Lambda)$, we denote by $K^\bullet$ also the complex in $\mathcal{D}^b(\Sm_{S,{\rm et}},\Lambda)$ that restricts to $p_X^\ast K^\bullet$ on each small site $X_{\rm et}$.
Further, we fix a Nisnevich local fibrant spectrum $E(K^\bullet)\in\Spt_{S^1}(\Sm_S)$ corresponding to $\mathbb{R}\varepsilon_\ast K^\bullet$ under the Dold--Kan correspondence.
\end{sect}

\begin{lemma}\label{lem: A1 local}
The spectrum $E(K^\bullet)$ is $\AA^1$-local.
\end{lemma}
\begin{proof}
Indeed, the projection $\pi\colon\mathbb{A}_X^1 \rightarrow X$ induces a quasi-isomorphism $p_X^\ast K^\bullet \rightarrow \mathbb{R}\pi_\ast\pi^\ast p_X^\ast K^\bullet$ (\eg~\cite[Corollary~7.7.4]{Fu11}) and hence a quasi-iso{\-}mor{\-}phism on cohomology
\begin{equation*}
  \mathbb{R}\Gamma(X_{\rm Nis}, \mathbb{R}\varepsilon_\ast p_X^\ast K^\bullet) \to
  \mathbb{R}\Gamma(\mathbb{A}_{X,{\rm Nis}}^1, \mathbb{R}\varepsilon_\ast p_{\mathbb{A}^1_X}^\ast K^\bullet).
\end{equation*}
Under the Dold--Kan correspondence this translates to our claim.
\end{proof}

\begin{sect}\label{para: what to do 1}
In order to apply Theorem \ref{thm: Bloch-Ogus abstract nonsense} to the  $\AA^1$-local spectrum $E(K^\bullet)$, we need to show that the forget support maps
\begin{equation*}
 E(K^\bullet)_{X_\sigma/X}(X_x^h) \to E(K^\bullet)_X(X_x^h)
\end{equation*}
vanish for all points $x$ in $X$.
Unravelling the definitions, these maps are just the forget support maps
\begin{equation*}
 \mathbb{R}\Gamma_{X_{x,\sigma}^h}(X_{x,{\rm et}}^h,p_X^\ast K^\bullet) \to \mathbb{R}\Gamma(X_{x,{\rm et}}^h,p_X^\ast K^\bullet)
\end{equation*}
of \'{e}tale cohomology.
\end{sect}

In the following, we will make use of Gabber's absolute purity theorem -- but not in its full strength.
The following easy special case will be sufficient for our cause:

\begin{lemma}\label{lem: absolute purity}
In the situation of \ref{para: situation etale Bloch-Ogus}, let $i\colon Z \hookrightarrow X$ be a closed subscheme of codimension $c$, contained in the special fibre of $X/S$.
Assume $Z/\mathbb{F}$ is smooth and connected.
Then the canonical morphism $\mathbb{R}i^!K^\bullet\vert_X \rightarrow K^\bullet\vert_Z(-c)[-2c]$ is a quasi-isomorphism for all l.c.c.~complexes $K^\bullet \in \mathcal{D}_b^c(S_{\rm et},\Lambda)$.
\end{lemma}
\begin{proof}
Say, $X/S$ and $Z/\mathbb{F}$ have relative dimension $m$ and $n$ respectively.
In particular, $c = m - n + 1$.
Consider the commutative diagram
\begin{equation*}
\begin{tikzcd}
Z \arrow[r, hook, "i"] \arrow[d, "p_Z"] & X\arrow[d, "p_X"]\\
\Spec(\mathbb{F}) \arrow[r, hook, "\sigma"] &S.
\end{tikzcd}
\end{equation*}
By Poincar\'{e}-duality for $X/S$ (respectively $Z/\mathbb{F}$) , $\mathbb{R}p_X^!K^\bullet \simeq p_X^\ast K^\bullet(m)[2m]$ (respectively $\mathbb{R}p_Z^!\sigma^\ast K^\bullet $ $\simeq p_Z^\ast\sigma^\ast K^\bullet(n)[2n]$).
Further, by the special case of absolute purity for the closed point in $S$ (which is an easy exercise -- e.g.~the proof of~\cite[Lemma~8.3.6]{Fu11} goes through unchanged for l.c.c.~sheaves and hence for l.c.c.~complexes), $\mathbb{R}\sigma^!K^\bullet \simeq \sigma^\ast K^\bullet[-2]$. 
Summing up, we get
\begin{equation*}
\begin{array}{rcl}
  \mathbb{R}i^!p_X^\ast K^\bullet&
  \simeq &
  \mathbb{R}i^! \mathbb{R}p_X^!K^\bullet(-m)[-2m]\phantom{,}
 \\
  &
  \simeq &
  \mathbb{R}p_Z^! \mathbb{R}\sigma^!K^\bullet(-m)[-2m]\phantom{,}
 \\
  &
  \simeq &
  \mathbb{R}p_Z^! \sigma^\ast K^\bullet(-m -1)[-2m-2]\phantom{,}
 \\
  &
  \simeq &
  p_Z^\ast\sigma^\ast K^\bullet(-c)[-2c],
\end{array}
\end{equation*}
finishing the proof.
\end{proof}

\begin{lemma}\label{lem: effaceability special fibre}
In the situation of \ref{para: situation etale Bloch-Ogus}, assume that $X$ is Henselian local with closed point $x$ in the special fibre of $X/S$.
Then the canonical morphism $\sigma_{X,\ast}\mathbb{R}\sigma_X^!\Lambda \rightarrow \Lambda$ induces the trivial morphism in $\mathcal{D}^b(k(x)_{\rm et},\Lambda)$:
\begin{equation*}
 x^\ast \sigma_{X,\ast}\mathbb{R}\sigma_X^!\Lambda \xrightarrow{\simeq ~0}
 x^\ast \Lambda
 \simeq \Lambda.
\end{equation*}
In particular, the canonical map $\mathbb{R}\Gamma_{X_\sigma}(X_{\rm et},\Lambda) \rightarrow \mathbb{R}\Gamma(X_{\rm et},\Lambda)$ is trivial.
\end{lemma}
\begin{proof}
The second claim follows from the first.
Indeed, as $X$ is local Henselian $\mathbb{R}\Gamma(X_{\rm et}, - ) \simeq \mathbb{R}\Gamma(k(x)_{\rm et},x^\ast(-))$.
For the first claim, it is enough to show that the Tate-twist
\begin{equation}\label{eq: forget support map twisted}
x^\ast( \sigma_{X,\ast}\mathbb{R}\sigma_X^!\Lambda(1) \to
 \Lambda(1))
\end{equation}
is trivial in $\mathcal{D}^b(k(x)_{\rm et},\Lambda)$.
By Lemma~\ref{lem: absolute purity}, $\mathbb{R}\sigma_X^!\Lambda(1) \simeq \Lambda[-2]$.
In particular, the sheaf-cohomology of~\eqref{eq: forget support map twisted} in degree $2$ is given by
\begin{equation}\label{eq: forget support map cohomology in degree 2}
 \Lambda = {\rm H}^0(X_{\sigma,{\rm et}},\Lambda) \simeq
  {\rm H}_{X_\sigma}^2(X_{\rm et},\Lambda(1)) \to
  {\rm H}^2(X_{\rm et},\Lambda(1))
 ,~
  1 \mapsto
  \hat{c}_1[\mathcal{O}(X_\sigma)],
\end{equation}
\ie, is trivial as $X$ is a local scheme.
Further,
\begin{equation*}
 x^\ast\sigma_{X,\ast}\mathbb{R}\sigma_X^!\Lambda(1) \simeq x^\ast\sigma_{X,\ast}\Lambda[-2] \simeq \Lambda[-2]
\end{equation*}
which implies
\begin{equation*}
 {\rm Hom}_{\mathcal{D}_c^b(k(x)_{\rm et},\Lambda)}(x^\ast \sigma_{X,\ast}\mathbb{R}\sigma_X^!\Lambda(1),x^\ast \Lambda(1)) \cong
 {\rm H}^2(k(x)_{\rm et},\Lambda(1))
\end{equation*}
and \eqref{eq: forget support map twisted} corresponds to a class contained in the image of \eqref{eq: forget support map cohomology in degree 2} (more precisely, \eqref{eq: forget support map twisted} corresponds to the class $\hat{c}_1[\mathcal{O}(X_\sigma)]$), hence it is trivial.
\end{proof}

\begin{corollary}\label{cor: effaceability special fibre}
In the situation of \ref{para: situation etale Bloch-Ogus}, assume that $X$ is Henselian local with closed point $x$ in the special fibre of $X/S$.
Let $K^\bullet \in \mathcal{D}_c^b(S_{\rm et},\Lambda)$ be a l.c.c.~complex.
Then the canonical morphism $\sigma_{X,\ast}\mathbb{R}\sigma_X^!K^\bullet\vert_X \rightarrow K^\bullet\vert_X$ induces the trivial morphism in $\mathcal{D}^b(k(x)_{\rm et},\Lambda)$:
\begin{equation*}
 x^\ast \sigma_{X,\ast}\mathbb{R}\sigma_X^!K^\bullet\vert_X \xrightarrow{\simeq ~0} 
 x^\ast K^\bullet\vert_X.
\end{equation*}
In particular, the canonical map $\mathbb{R}\Gamma_{X_\sigma}(X_{\rm et},K^\bullet\vert_X) \rightarrow \mathbb{R}\Gamma(X_{\rm et},K^\bullet\vert_X)$ is trivial.
\end{corollary}
\begin{proof}
By Lemma \ref{lem: effaceability special fibre}, $x^\ast$ applied to
\begin{equation}\label{eq: forget support map tensored}
 (
  \sigma_{X,\ast}\mathbb{R}\sigma_X^!\Lambda \to
  \Lambda
 )\otimes^{\mathbb{L}} p_X^\ast K^\bullet
\end{equation}
is trivial.
By the projection formula and Lemma \ref{lem: absolute purity}, \eqref{eq: forget support map tensored} is isomorphic to the canonical morphism $\sigma_{X,\ast}\mathbb{R}\sigma_X^! p_X^\ast K^\bullet \rightarrow p_X^\ast K^\bullet$, so the claim follows.
\end{proof}

Combining Theorem~\ref{thm: Bloch-Ogus abstract nonsense} and Corollary~\ref{cor: effaceability special fibre}, we get:

\begin{theorem}\label{thm: Bloch-Ogus etale cohomology}
Let $S$ be the spectrum of a Henselian discrete valuation ring with infinite residue field $\mathbb{F}$.
Let $X/S$ be smooth, $d = {\rm dim}(X)$ and $K^\bullet$ an l.c.c.~complex in $\mathcal{D}_c^b(S_{\rm et},\Lambda)$.
Then the Nisnevich Gersten complex $\mathcal{G}^\bullet(E(K^\bullet),n)$ is a flasque resolution of the Nisnevich sheafification $\mathbb{R}^n\varepsilon_*K^\bullet\vert_X$ of \'{e}tale cohomology with coefficients~$K^\bullet$.
In particular, we get the exact sequence
\begin{multline*}
   0 \rightarrow \mathbb{R}^n\varepsilon_\ast K^\bullet\vert_X \rightarrow \bigoplus_{z\in X^{(0)}}z_\ast {\rm H}^{n}(k(z),K^\bullet\vert_{k(z)})\rightarrow\dots\\
  \dots \rightarrow \bigoplus_{z\in X^{(d)}}z_\ast {\rm H}^{n-d}(k(z),K^\bullet\vert_{k(z)}(-d)) \rightarrow 0.
  \end{multline*}
\end{theorem}
\begin{proof}
The spectrum $E(K^\bullet)$ is $\AA^1$-local by Lemma \ref{lem: A1 local}.
Combining Theorem~\ref{thm: Bloch-Ogus abstract nonsense} and Corollary~\ref{cor: effaceability special fibre}, we get that $\mathcal{G}^\bullet(E(K^\bullet),n)$ is a flasque resolution of $\mathbb{R}^n\varepsilon_*K^\bullet\vert_X$.
\\
Let us compute $\mathcal{G}^s(E(K^\bullet),n) = \bigoplus_{z\in X^{(s)}}\mathfrak{z}_\ast\mathfrak{z}^\ast E(K^\bullet)_{Z/X}^{n+s}$:
In the proof of Proposition~\ref{proposition:flasque} we saw that $\mathfrak{z}^\ast E(K^\bullet)_{Z/X}^{n+s} = z_\ast z^\ast E(K^\bullet)_{z/X_z^{\rm loc}}^{n+s}$.
Unravelling the definitions, $z^\ast E(K^\bullet)_{z/X_z^{\rm loc}}^{n+s} \cong {\rm H}_z^{n+s}(X_{z,{\rm et}}^{\rm loc},K^\bullet)$.
By absolute purity, ${\rm H}_z^{n+s}(X_{z,{\rm et}}^{\rm loc},K^\bullet) \cong {\rm H}^{n-s}(k(z),K^\bullet(-s))$, which finishes the proof.
\end{proof}

\begin{remark}\label{remark:avoidpurity}
We can avoid absolute purity in its full strength if we assume $k$ and $\mathbb{F}$ to be perfect:
Computing $\mathfrak{z}^\ast E(K^\bullet)_{Z/X}^{n+s}$ under this assumption, we may assume $Z$ to be smooth over $k$ (if $z$ is contained in the generic fibre of $X/S$) or smooth over $\mathbb{F}$ (if $z$ is contained in the special fibre of $X/S$) by generic smoothness.
In both cases, $\mathfrak{z}^\ast E(K^\bullet)_{Z/X}^{n+s} \cong z_\ast H^{n-s}(k(z),K^\bullet(-s))$, either by relative purity or by Lemma~\ref{lem: absolute purity}.
\end{remark}

Taking Nisnevich stalks, we get:

\begin{corollary}\label{cor: Bloch Ogus for etale cohomology of Nisnevich local schemes}
Let $S$ be the spectrum of a Henselian discrete valuation ring with infinite residue field $\mathbb{F}$.
Let $X/S$ be smooth of finite type, $d = {\rm dim}(X)$ and $K^\bullet$ an l.c.c.~complex in $\mathcal{D}_c^b(S_{\rm et},\Lambda)$.
Let $x$ be a point of $X$ and $Y = X_x^{h}$ the Nisnevich local scheme at $x$.
Then there is an exact sequence
\begin{multline*}
 0 \to {\rm H}^n(Y_{{\rm et}},K^\bullet\vert_Y) \xrightarrow{e} \bigoplus_{z \in Y^{(0)}} {\rm H}^n(k(z),z^\ast K^\bullet\vert_Y) \xrightarrow{d^{0}} \cdots\\
 \cdots \xrightarrow{d^{d-1}} \bigoplus_{z \in Y^{(d)}} {\rm H}^{n-d}(k(z),z^\ast K^\bullet\vert_Y(-d)) \to 0 .
\end{multline*}
\end{corollary}

\begin{remark}\label{remark:geisser}
Using the Bloch--Kato-Conjecture, in \cite{Geisser04} Geisser proved the exactness of the Gersten complex in degree $n$ for $X/S$ smooth even for $\varepsilon\colon X_{\rm et}\rightarrow X_{\rm Zar}$, but only for coefficients $K^\bullet = \Lambda(r) = \boldsymbol{\mu}_{m}^{\otimes r}$ for $n \leq r$.
If $S$ is not strictly Henselian, this assumption excludes ${\rm H}^{2r}(Y_{\rm et},\Lambda(r))$ for $r>0$, \ie, the targets of the cycle class maps.
\end{remark}

\end{document}